\documentclass[3p,twocolumn,authoryear]{elsarticle}

\usepackage{microtype}

\usepackage{bbm}
\usepackage[utf8]{inputenc}
\usepackage[T1]{fontenc}
\usepackage{lmodern}
\usepackage{graphicx,color}
\usepackage{caption}
\usepackage{subcaption}
\usepackage{mathrsfs,amssymb,amsfonts,amsthm,amsmath,amsbsy}
\usepackage{appendix}

\theoremstyle{plain} 
\newtheorem{thm}{Theorem}[section] 

\theoremstyle{definition}

\newtheorem{lem}{Lemma}[section] 

\theoremstyle{remark} 
\newtheorem{oss}{Remark}[section]

\allowdisplaybreaks

\begin{document}

\title{On dynamic random graphs with degree homogenization via anti-preferential attachment probabilities}


\author[1]{Umberto De Ambroggio}
\ead{uda22@bath.ac.uk}
\author[2]{Federico Polito\corref{cor1}}
\ead{federico.polito@unito.it}
\author[2]{Laura Sacerdote}
\ead{laura.sacerdote@unito.it}

\address[1]{Department of Mathematical Sciences, University of Bath, UK}
\address[2]{Mathematics Department ``G.~Peano'', University of Torino, Italy}

	\begin{abstract}
		We analyze a dynamic random undirected graph in which newly added vertices are connected
		to those already present in the graph
		either using, with probability $p$, an anti-preferential attachment mechanism or, with probability $1-p$, a preferential attachment mechanism. We derive the asymptotic degree distribution in the general case and study the asymptotic behaviour of the expected degree process in the general and that of the degree process in the pure anti-preferential attachment case. Degree homogenization mainly affects convergence rates for the former case and also the limiting degree distribution in the latter.
		Lastly, we perform a simulative study of a variation of the introduced model allowing for anti-preferential attachment
		probabilities given in terms of the current maximum degree of the graph.
	\end{abstract}

	\maketitle

	\section{Introduction}
	
		The appearance of the seminal paper of \cite{MR2091634} ignited a vast interest for preferential-attachment models of random graphs in many different fields of research. In the early 2000s, \cite{MR1824277} gave a rigorous mathematical formulation of the model described by \cite{MR2091634}. The literature on dynamic graphs governed by a preferential-attachment mechanism is vast and it is impossible to give a complete list of relevant references. However, the reader can refer to the book by \cite{VDH} and to \cite{newman2006structure} and \cite{MR2010377}, which would give an excellent overview. For generalizations of the basic preferential-attachment model of Barab\'asi and Albert see for instance \cite{MR1824277}, \cite{krapivsky2000connectivity}, \cite{krapivsky2001organization}, \cite{MR2522950}, \cite{MR2511283}, \cite{MR3059201}, \cite{betken2019fluctuations}, \cite{MR3076683}, and the book by \cite{VDH}.
		
		The preferential attachment paradigm has been heavily advocated as an explanation of the fact that degree distributions in real-world networks are well-described by power-laws (even if this might not be always the case: see for instance \cite{broido2019scale}).
		From a modeling perspective, classical preferential attachment random graph models implicitly assume that all vertices behave in a similar way, meaning that all of them tend to connect to vertices of high degree.
		However, recently there has been interest in random graphs models relaxing the classical preferential-attachment hypothesis but still exhibiting power-law degree distributions. In those models, newly added vertices are allowed to attach their edges by means of a mixture of uniform and preferential attachment mechanisms (see \cite{COOP} and \cite{PACH}).
		Apart from the possibility of creating more accurate models for real world networks, the idea of growing a graph using a mixture of uniform and preferential attachment mechanisms has the effect of producing a homogenization of vertices degrees. In other terms, the fact that some of the vertices will connect to older nodes chosen according to  a uniform distribution has (at least at an intuitive level) the effect of slowing down the growth of the degree of highly connected vertices, thus balancing the number of connections for different nodes. This idea of producing a homogenization of vertices degree has been further pursued by \cite{MR2119994} and more recently by \cite{MR2582718}, \cite{MR2770911} (see also \cite{lansky2014role}, \cite{MR2314629}, \cite{MR2783864} for models implementing edge deletion, and \cite{MR2896674,MR2609453} for more general models). In these models, at each discrete time step there is a chance that a vertex (selected uniformly at random from the set of vertices) is removed from the graph. If the vertex which has been selected uniformly at random is a hub, then many vertices (all vertices connected to it) will face a reduction of their degree (by a quantity corresponding to the number of edges that there were between them and the hub in case of a multigraph). Therefore, in this sense the removal of a highly connected vertex produces a local homogenization of the vertices degree. In the model appearing in \cite{MR2770911}, at each discrete time step $t$, the following three different things can occur. Provided that the number of vertices at time $t-1$ is larger or equal than unity:

		\begin{itemize}
			\item with probability $p_1$, a vertex is added and connected to a random number of older vertices selected uniformly at random from the set of vertices present in $G_{t-1}$, the graph at time $t-1$;
			\item with probability $p_2 \le 1-p_1$ a vertex is sampled uniformly at random from $G_{t-1}$, copied and added to the graph together with all its connections;
			\item with probability $p_3 = 1-p_1-p_2$  a vertex is sampled uniformly at random from $G_{t-1}$ and deleted from the graph together with all its connections.
		\end{itemize}

		The authors show that, if ``copying'' occurs more frequently than ``deletion'', then the graph exhibits a power-law degree distribution with exponent larger than unity. Viceversa, if ``deletion'' is more likely to occur than ``copying'', then the resulting graph exhibits a degree distribution of exponential type. An intuitive explanation of this fact is that in case deletion occurs very frequently then it is likely that from time to time a hub is selected for removal, thus reducing the degree of many of the vertices in the graph. Therefore, when deletion occurs rather frequently then a strong degree homegenization occurs. 
	
		For the sake of completeness, we lastly mention that a further model belonging to the class mentioned above is due to \cite{johansson2016deletion}.
				
		In the present paper we aim at capturing this homogenization effect of vertices degree, but without any vertex deletion. In particular, we study a dynamic random undirected graph model in which every newly added vertex is connected to $m\in \mathbb{N}^*=\{1,2,\dots\}$ random older vertices, either using an anti-preferential attachment or a preferential attachment mechanism, the former meaning that newly added vertices are more likely to connect to vertices of low degree. More precisely, at each discrete-time a new vertex is added to the graph. With probability $p\in [0,1]$ it selects the $m$ target vertices according to a suitable anti-preferential attachment rule (see the definition of the model in Section \ref{iniz}), and with probability $1-p$ using a classical preferential attachment mechanism.
		
		The pure anti-preferential attachment and the pure preferential attachment case can be recovered by setting $p=1$ and $p=0$, respectively.
		The simultaneous presence of a random fraction of the vertices chosen by the above mentioned mechanisms reflects in the change of the exponent of the degree distribution of the model.
		Specifically, the asymptotic degree distribution of such random graph has a right tail decaying as a power-law with exponent $(p-3)/(1-p)$, $p \in [0,1)$.
		Interestingly enough, for the definition of anti-preferential attachment mechanism that we propose (see Section \ref{iniz}), the mixed preferential attachment - anti-preferential attachment model (PA-APA model in the following), although defined in a different way, leads to a degree distribution resembling that in \cite{COOP}. Indeed, in the pure anti-preferential attachment regime the model tends to produce a homogenization of vertices' degree, similarly to a uniform attachment. In other terms, after a substantial amount of vertices have been added to the graph, the probabilistic rules (governing the attachment mechanisms of newly added vertices) in the mixed PA-UA model and the PA-APA model will be close. Hence, the model presented here can be seen as an alternative to the model of Cooper and Frieze to model phenomena presenting a limiting power-law degree distribution as that appearing, for suitably chosen parameters, in their paper.
		
		Moreover, we also analyze by means of simulations a variation of the anti-preferential attachment mechanism in which
		probabilities are calculated with respect to the current maximum degree of the graph process instead of with respect of
		the maximum theoretical degree.
		
Finally, before turning to outlining the paper organization, we recall that an anti-preferential attachment mechanism already appeared in the context of random mappings (\cite{MR2428980}) and in the recent paper
		by \cite{sendina2016assortativity}. Regarding the latter paper, the authors propose a method which generates networks characterized by scale-free degree distributions and tunable assortativity. The latter is a feature found in many online social and neural networks (see e.g.\ \cite{viswanath2009evolution},
		\cite{teller2014emergence} and \cite{de2014emergence}) which stands for the tendency of vertices to establish connections with other vertices of similar degree. The novelty of their approach is that they avoid to introduce assortativity in a pre-generated network by means of degree-preserving link permutations. In this way they  preserve intrinsic relevant characteristic of the network itself that otherwise would have been destroyed.
In their model, there are two distinct populations of vertices which are incrementally added to
an initial network by selecting a subgraph to connect to at random. One population follows preferential attachment, while the other connects via anti-preferential attachment: they link to lower degree vertices when added to the network.
More specifically, the model produces a network with $N$ vertices, constructed sequentially in discrete time by adding one vertex at a time.
The graph starts with a clique of $m \leq N_0 \ll N$ vertices, and at each time $1\leq t\leq N - N_0$ a vertex is added to the graph and connected to $m$ older vertices according to the following procedure:
\begin{enumerate}
	\item an \emph{anchor} vertex $j$ is selected uniformly at random from the set of vertices present at time $t-1$;
	\item a subgraph $G_j$  is constructed by including $j$ and all other vertices at graph distance less than or equal to $l$ from $j$ (in the paper $l=1$);
	\item with probability $1-p$ the new vertex selects $m$ vertices from the subgraph $G_j$ uniformly at random and connects to them; with probability $p$, the new vertex connects to the $m$ vertices of lowest degree in $G_j$.
\end{enumerate}
The resulting network exhibits a scale-free nature and presents a fully controllable level of global assortativity (which depends on $p$ and $m$).

Although different, the graph construction rules are similar in spirit to those at the basis of the definition of the PA-APA model we study in this paper (see Section \ref{iniz}).

Returning to the present paper, its structure is the following: Section \ref{iniz} describes the model of interest and suggests some generalizations. Section \ref{main} contains the main results. We start the analysis in Section \ref{iniz2}, where we analyze the asymptotic degree distribution which turns out to have a power-law tails with exponent $(p-3)/(1-p)$ (see Theorem \ref{propa} and Remark \ref{propa2}). This shows that even graphs growing with a mixed mechanism in which the anti-preferential attachment regime is dominant, have power-law degree distributions. Furthermore, by tuning the value of $p$, every exponent in $(-\infty,-3]$ can be generated. This fact
		also confirms that the robustness of the scale-free nature of the graph to important changes in the attachment mechanism,
		as shown in \cite{COOP} and \cite{PACH}, is here retained.
		In other words, a graph growing in large part according to an anti-preferential attachment mechanism still shows a power-law degree distribution if we let an infinitesimal possibility to the preferential attachment mechanism when building connections between vertices. The subsequent Sections \ref{iniz3} and \ref{iniz4} concern the asymptotic behaviour of the degree process for each given vertex.
		First, it is established (Theorem \ref{esta}, Section \ref{iniz3}) that the expected degree process is controlled by $t^{(1-p)/2}\ln t^p$, and then the almost sure convergence for the degree process is investigated in the
		special case of pure anti-preferential attachment (Section \ref{iniz4}).
		Section \ref{variation} concludes the paper presenting a simulation analysis of a variation of the PA-APA model
		where the probabilities are calculated with respect to the current maximum degree of the graph process.
		Finally, to make the paper more readable, all the proofs are postponed to Appendix \ref{app}.

	\section{PA-APA model}\label{iniz}
		
		Let $m\in \mathbb{N}^*$, $t \in \mathbb{N}$. The model we investigate produces a graph sequence that we denote by $\{G_t\colon t \in \mathbb{N}
		\}$ and which, for every time $t$, yields a graph of $t$ vertices and $mt$ edges. For $t=0$, $G_0$ is the
		empty graph and for $t \ge 1$ let us denote the vertices in $G_t$ by $v_1,\dots,v_t$. Let $\mathcal{F}_0$ be the trivial sigma-algebra and denote $\mathcal{F}_t$ the $\sigma$-algebra generated by the graph process up to time $t\geq 1$; more precisely,
		\begin{equation*}
			\mathcal{F}_t=\sigma(\{G_s\colon 1\leq s\leq t\}), \qquad t\geq 1.
		\end{equation*}
		For every $i\in \mathbb{N}^*$ we denote by $d(v_i,t)$ the degree of vertex $v_i$ at time $t \ge i$. We set $d(v_i,t)=0$ if $t < i$. The random graph process $(G_t)_{t\geq 1}$ evolves as follows. 

		Let $(Y_{t})_{t\geq 2}$ be a sequence of i.i.d.\ Bernoulli random variables of parameter $p \in [0,1]$
		independent of $(G_t)_{t\geq 1}$. Let $G_1$ be a graph consisting of a single vertex $v_1$ with $m$ self-loops. For every $t\in \mathbb{N}^*$, to construct $G_{t+1}$ from $G_t$, add a new vertex $v_{t+1}$ and then add $m$ edges between $v_{t+1}$ and vertices of $G_t$ (therefore vertices added at times $t\geq 2$ do not have self-loops). The $m$ target vertices in $G_t$ are chosen according to the following procedure, which admits multiple edges between distinct vertices.
		\begin{itemize}
			\item If $Y_{t+1}=0$ (which happens with probability $1-p$) we select $m$ random vertices $W_{t+1}^1,\dots,
				W_{t+1}^m$ from $G_t$ according to the preferential attachment mechanism
				\begin{equation}
					\label{prefa}
					\mathbb{P}(W_{t+1}^r=v_i|\mathcal{F}_t)=\frac{d(v_i,t)}{2mt},
				\end{equation}
				where $1\leq i\leq t$, independently for each $r\in \{1,\dots,m\}$.
			\item If $Y_{t+1}=1$ (which happens with probability $p$) we select $m$ random vertices $W_{t+1}^1,\dots, W_{t+1}^m$
				from $G_t$ according to the anti-preferential attachment mechanism
				\begin{equation}
					\label{aprefa}
					\mathbb{P}(W_{t+1}^r=v_i|\mathcal{F}_t)=\frac{2mt+1-d(v_i,t)}{t(2mt+1-2m)},
				\end{equation}
				where $1\leq i\leq t$, independently for each $r\in \{1,\dots,m\}$.
		\end{itemize}
	
		\begin{oss}
			Formula \eqref{aprefa} shows that (when $Y_{t+1}=1$) vertices with smaller degree are selected with higher probability by newly added vertices. This seems to be the simplest possible anti-preferential attachment mechanism one can imagine, and intuitively this formulation yields (for large $t$) a homogenization of the degrees similar to that induced by uniform attachment. This is shown in Theorem 2.1 below (case $p=1$). Indeed we will see that in the pure anti-preferential attachment regime the resulting degree distribution is geometric, which is the same degree distribution appearing in a uniform attachment model (see for instance \cite{MR3968514}).
		\end{oss}
		\begin{oss}
			The process $(Y_t)_{t\geq 2}$ encodes the information concerning the attachment mechanism chosen by vertices when added to the graph: at each time $t$ we add a vertex $v_{t+1}$ to $G_t$, we generate an independent Bernoulli r.v.\ $Y_{t+1}$ and:
			\begin{itemize}
				\item if $Y_{t+1}=1$ then we attach $m$ edges between $v_{t+1}$ and vertices of $G_t$, selected according to the
					anti-preferential attachment mechanism \eqref{aprefa};
				\item if $Y_{t+1}=0$ then we attach $m$ edges between $v_{t+1}$ and vertices of $G_t$, selected according to the
					preferential attachment mechanism \eqref{prefa}.
			\end{itemize}
			This is the most basic example in which the process $(Y_t)_{t\geq 2}$ models the way in which the
			two possible regimes coexist in the process' dynamics. More general cases taking into account
			different characteristics could be further considered. For instance a dependence structure in the
			choice process $Y_t$ could model the case in which the newly added vertex regime is somehow affected by
			the previous choices.
			Furthermore, formulas \eqref{prefa} and \eqref{aprefa} are both special cases of a general attachment rule in which the
			target vertices are chosen according to the probabilities $\mathbb{P}(W_{t+1}^r=v_i|\mathcal{F}_t)=h(d(v_i,t))/\sum_{j=1}^t h(d(v_j,t))$, $1 \le i \le t$, $r\in \{1,\dots,m\}$, where $h$ is strictly positive and bounded.
		\end{oss}

	\section{Main results}\label{main}

		\subsection{Degree distribution}\label{iniz2}

			We start by introducing some notation. For $m\leq k\leq 2mt$,
			and $m,t \in \mathbb{N}^*$, we define:
			\begin{itemize}
				\item $H(k,t+1)=p\frac{2mt+1-k}{t(2mt+1-2m)}+(1-p)\frac{k}{2mt}$, the unconditional attachment
					probability with any rule to a given vertex with degree $k$ at time $t$;
				\item $K(k,t+1)=1-H(k,t+1)$;
				\item $N_k(t)=\sum_{j=1}^{t}\mathbbm{1}_{\{d(v_j,t)=k\}}$, the number of vertices of $G_t$ with degree $k$;
				\item $P(k,t)=\mathbb{E}[N_k(t)/t]$, the expected proportion of vertices of $G_t$ with degree $k$;
				\item $Q(k,t)=\frac{1}{t-1}\sum_{j=2}^{t}\mathbb{P}(d(v_j,t)=k)$, $t\geq 2$;
				\item $f(v_i,k,t)=\mathbb{P}\Bigl(d(v_i,t)=k,d(v_i,s)\neq k \hspace{0.15cm} \forall s=i,\dots, t-1\Bigr)$,
					and $f(v_i,k,i) = \delta_{km}$, where $i \ge 2$, $m \le k \le (t-i+1)m$, $t\geq i+1$, and $\delta_{km}$ is the Kroenecker's delta.
					The function $f$ is the probability
					that vertex	$v_i$ has degree $k$ for the first time at time $t$.
			\end{itemize}
			Henceforth we adopt the conventions that empty products equal unity and empty sums equal zero.
			
			\begin{thm}
				\label{propa}
				Let $m\in \mathbb{N}^*$ and write $P(k)= \lim_{t\rightarrow \infty}P(k,t)$.
				If $p\in [0,1)$,
				\begin{equation*}
					P(k) = \frac{2}{2+m+mp}\cdot \frac{\xi(k)}{\xi(m)}, \qquad k \ge m,
				\end{equation*}
				with $\xi(k) = \Gamma\left(k+2m\frac{p}{1-p}\right)/\Gamma\left(k+1+\frac{2(m+1)}{1-p}\right)$ and where $\Gamma$ is the Euler's Gamma function.

				If $p=1$,
				\begin{equation*}
					P(k)=\frac{1}{m+1}\left(1-\frac{1}{m+1}\right)^{k-m}, \qquad k\geq m.
				\end{equation*}
			\end{thm}

			\begin{oss}\label{propa2}
				Using the well-known asymptotic expansion of ratio of gamma functions (\cite{MR0043948}) we easily derive that, if $p \in [0,1)$,
				\begin{align}
					P(k) \sim k^{\frac{p-3}{1-p}}
				\end{align}
				for large values of $k$.
				No matter how large we choose $p$ (i.e.\ how likely new added vertices attach their edges using an anti-preferential attachment rule), the degree distribution of the mixed model remains a power-law.
				Hubs, once formed, remain such, thus producing a heavy tail in the degree distribution.	
				Moreover, notice that the mixed PA-APA model is able to recover any power-law with exponent in $(-\infty,-3]$ simply by modifying the parameter $p$.
			\end{oss}

			\begin{oss}[Convergence in probability of $N_k(t)/t$]
				Define the Doob martingale $(M_n)_{0 \le n \le t}$ such that $M_n=\mathbb{E} [N_k(t)|\mathcal{F}_n]$
				(see \cite{MR0002052}). Note that $|M_n-M_{n-1}|\le mt$ for each $0 \le n \le t$. By applying
				Azuma--Hoeffding inequality we get the bound $\mathbb{P}(|M_t-M_0|\ge \sqrt{t\ln t})\le \exp (-\ln t /(8m^2))$.
				Then, letting $t \to \infty$ and considering Theorem \ref{propa}, we obtain
				\begin{align*}
					N_k(t)/t \overset{\mathbb{P}}{\to}P(k), \qquad k \ge m.
				\end{align*} 
			\end{oss}

		\subsection{Asymptotic behavior of the expected degree of a given vertex}\label{iniz3}
				
			We move now to analyzing the rate of divergence of the expected degree $\mathbb{E}[d(v_i,t)]$, for fixed $i\in \mathbb{N}^*$, as $t\rightarrow \infty$. In order to simplify the notation we introduce two quantities which will appear often in the computations involving the expected degree of a given vertex:
			\begin{align*}
				c_l=\frac{2ml+1}{l(2ml+1-2m)}, \qquad e_l=\frac{m}{l(2ml+1-2m)},
			\end{align*}
			where $l \in \mathbb{N}^*$.
			\begin{thm}
				\label{waif}
				Let $m\in \mathbb{N}^*$, $p\in [0,1]$. Then for every $t\geq i$,
				\begin{align}
					\label{wwaif}
					\mathbb{E}[d(v_i,t)]&=(1+\delta_{i1})m\prod_{l=i}^{t-1}C(l,p,m)\notag\\
					&+\sum_{l=i}^{t-1}m\,p\,c_l\prod_{h=l+1}^{t-1}C(h,p,m),
				\end{align}
				where $C(l,p,m)=\left(1+\frac{1-p}{2l}-p\, e_l\right)$ and where $\delta_{ij}$ is the Kroenecker's delta.
			\end{thm}

			\begin{oss}
				\label{sqrt}
				Let $t\geq i$. Notice that when $p=0$ (i.e.\ in the pure preferential attachment model) the expected degree
				\eqref{wwaif} reduces to
				\begin{equation*}
					\mathbb{E}[d(v_i,t)]
					=(1+\delta_{i1})m\frac{\Gamma(i)\Gamma(t+1/2)}{\Gamma(i+1/2)\Gamma(t)}.
				\end{equation*}
				Exploiting the fact that $\Gamma(t+1/2)/\Gamma(t) = \sqrt{t}(1+O(1/t))$, it is immediate to see that, for every $i\in \mathbb{N}^*$,
				\begin{equation*}
					\lim_{t \rightarrow\infty}
					\mathbb{E}\left(\frac{d(v_i,t)}{\sqrt{t}}\frac{\Gamma(i+1/2)}{\Gamma(i)} \right)
					=(1+\delta_{i1})m.
				\end{equation*}
				Therefore, in the preferential attachment case, the expected degree of any given vertex grows as $\sqrt{t}$ (see e.g.\ \cite{VDH}, Chapter 8).
				On the other hand, when $p=1$ (i.e.\ in the pure anti-preferential attachment model), from \eqref{wwaif},
				\begin{align*}
					\mathbb{E}[d(v_i,t)] &= (1+\delta_{i1})m\prod_{l=1}^{t-1}\left(1-e_l\right) \\
					&+ \sum_{l=1}^{t-1}m\,c_l\prod_{h=l+1}^{t-1}\left(1-e_h\right).
				\end{align*}
				
				Next theorem and the following Remark \ref{loo} show that, in the pure anti-preferential regime,
				the expected degree of a vertex grows as $\ln t$.
			\end{oss}

			\begin{thm}
				\label{sqrtno}
				Let $m\in \mathbb{N}^*$ and $i\geq 2$. In the pure anti-preferential attachment model $(p=1)$ we have that
				\begin{equation}
					\label{cali}
					\lim_{t \rightarrow \infty}\mathbb{E}\left(\frac{d(v_i,t)}{\sum_{l=i}^{t-1}\ln\left(1+c_l\right)
					\prod_{h=l+1}^{t-1}\left(1-e_h\right)}\right)=m.
				\end{equation}
			\end{thm}

			\begin{oss}
				\label{loo}
				To understand better the rate of divergence of the degree process, observe that the denominator of \eqref{cali}
				has the following limiting behaviour. Since
				\begin{align*}
					&\prod_{h=i}^{t-1}\left(1-e_h\right)\sum_{l=i}^{t-1}\ln\left(1+c_l\right)\\
					&\leq \sum_{l=i}^{t-1}\ln\left(1+c_l\right)\prod_{h=l+1}^{t-1}\left(1-e_h\right)
					\leq \sum_{l=i}^{t-1}\ln\left(1+c_l\right),
				\end{align*}
				we have for every $t\geq i+1$,
				\begin{align*}
					0&<\prod_{h=i}^{t-1}\left(1-e_h\right)\\
					&\leq \frac{\sum_{l=i}^{t-1}\ln\left(1+c_l\right)\prod_{h=l+1}^{t-1}\left(1-e_h\right)}{\sum_{l=i}^{t-1}\ln\left(1+c_l\right)}\leq 1.
				\end{align*}
				Thus
				\begin{align*}
					0&<\liminf_{t\rightarrow \infty}\frac{\sum_{l=i}^{t-1}\ln\left(1+c_l\right)\prod_{h=l+1}^{t-1}\left(1-e_h\right)}{\sum_{l=i}^{t-1}\ln\left(1+c_l\right)}\\
					&\leq \limsup_{t\rightarrow \infty}\frac{\sum_{l=i}^{t-1}\ln\left(1+c_l\right)\prod_{h=l+1}^{t-1}\left(1-e_h\right)}{\sum_{l=i}^{t-1}\ln\left(1+c_l\right)}\leq 1,
				\end{align*}
				so that
				\begin{align*}
					&\sum_{l=i}^{t-1}\ln\left(1+c_l\right)\prod_{h=l+1}^{t-1}\left(1-e_h\right)\\	&=\Theta\left(\sum_{l=i}^{t-1}\ln\left(1+c_l\right)\right)=\Theta(\ln t),
				\end{align*}
				as $\sum_{l=i}^{t-1}\ln\left(1+1/l\right)=\ln t-\sum_{l=1}^{i-1}\ln\left(1+1/l\right)$.
			\end{oss}
		
			Therefore, from Remarks \ref{sqrt} and \ref{loo}, in the classical preferential attachment model ($p=0$) the expected degree of any given vertex grows as $\sqrt{t}$, whereas in the pure anti-preferential attachment case ($p=1$) it grows as $\ln t$. Next result tells us that, in the mixed PA-APA model ($p \in (0,1)$), the growth of the expected degree is controlled by $t^{(1-p)/2}\ln t^p$.
			
			\begin{thm}\label{esta}
				In the mixed PA-APA model of parameter $p\in (0,1)$ we have that $\mathbb{E}[d(v_i,t)]=O\left(t^{(1-p)/2}\ln t^p\right)$, $i \ge 2$.
			\end{thm}

		\subsection{Almost sure convergence of the degree process}\label{iniz4}
			
			Next we discuss the almost sure convergence of $d(v_i,t)$ when $p\in \{0,1\}$. Concerning the case $p=0$, that is in the pure preferential attachment case, it is well known (see e.g.\ \cite{VDH}, Chapter 8) that $t^{-1/2}d(v_i,t)\overset{a.s.}{\longrightarrow}\eta_i$,
			where, for each given $i$, $\eta_i$ is an almost surely positive random variable with finite mean. In other terms, the degree of any given vertex $v_i$ grows as $\sqrt{t}$ (i.e.\ $d(v_i,t)=O(\sqrt{t})$ as $t\rightarrow \infty$).
			If $p=1$, that is for the pure anti-preferential case we have the following results.
			
			First we show that the degree of each vertex grows much slower than $\ln^s t$ for every $s>1$ (i.e.\ $d(v_i,t)=o\left(\ln^s t\right)$ as $t\rightarrow \infty$), and hence much slower than in the usual preferential attachment setting.

			\begin{thm}
				\label{elvira}
				Let $p=1$ (anti-preferential case) and let $m,i\in \mathbb{N}^*$. Then, for every fixed $s>1$,
				\begin{align*}
					\lim_{t \rightarrow \infty}\frac{d(v_i,t)}{\left(\sum_{l=i}^{t-1}\ln\left(1+c_l\right)\right)^s}=0, \qquad \text{a.s.}
				\end{align*}
			\end{thm}

			Then, next theorem shows that in the pure anti-preferential attachment case (i.e.\ when $p=1$) the growth of $d(v_i,t)$ for any given $i$ is controlled by $\ln t $.

			\begin{thm}
				\label{lntlnt}
				Let $m,i\in \mathbb{N}^*$. In the pure anti-preferential attachment model we have that
				\begin{align*}
					\limsup_{t\rightarrow \infty}\frac{d(v_i,t)}{\ln t}< \infty, \qquad \text{a.s.}
				\end{align*}
			\end{thm}

			Lastly, to gain more insight on the rate of growth of the degree process we
			state the following theorem.
			\begin{thm}
				\label{elvira2}
				Let $p=1$, $i > 1$, and $\gamma(t)$ be a positive function such that $\gamma(t) \to \infty$ as $t \to \infty$ and $\sum_{j=1}^t(j \ln j)^{-1}/\gamma(t) \to 0$.
				Then $d(v_i,t)\le \kappa \, \gamma(t) \ln t$ with high probability as $t \to \infty$, where $\kappa>1$ is a constant.
			\end{thm}

			\begin{oss}
				Notice that when $\gamma(t)$ grows faster than $\ln^\alpha t$, $\alpha > 0$, Theorem \ref{elvira2} is weaker than
				Theorem \ref{elvira}. As an example of application in the non-trivial case one might consider $\gamma(t) = (\ln \ln t)^\beta$, $\beta\ge 2$.
			\end{oss}
			
	\section{A variation of the PA-APA model}\label{variation}

		\begin{figure}
				\begin{subfigure}{.23\textwidth}
					\includegraphics[scale=0.22]{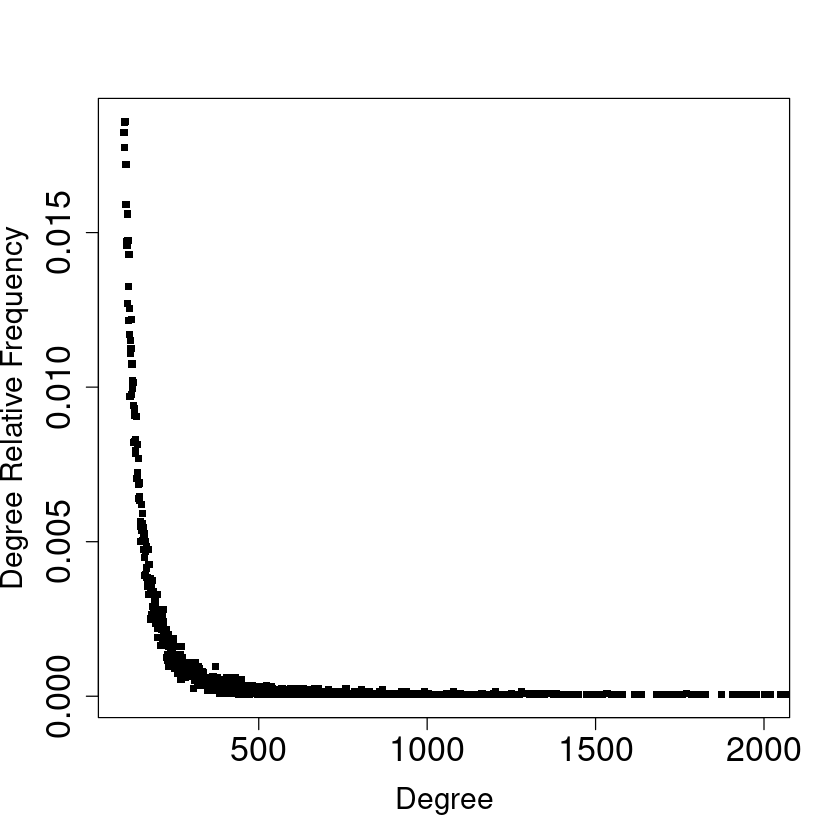}
					\subcaption{$p=0$}
				\end{subfigure}
				\begin{subfigure}{.23\textwidth}
					\includegraphics[scale=0.22]{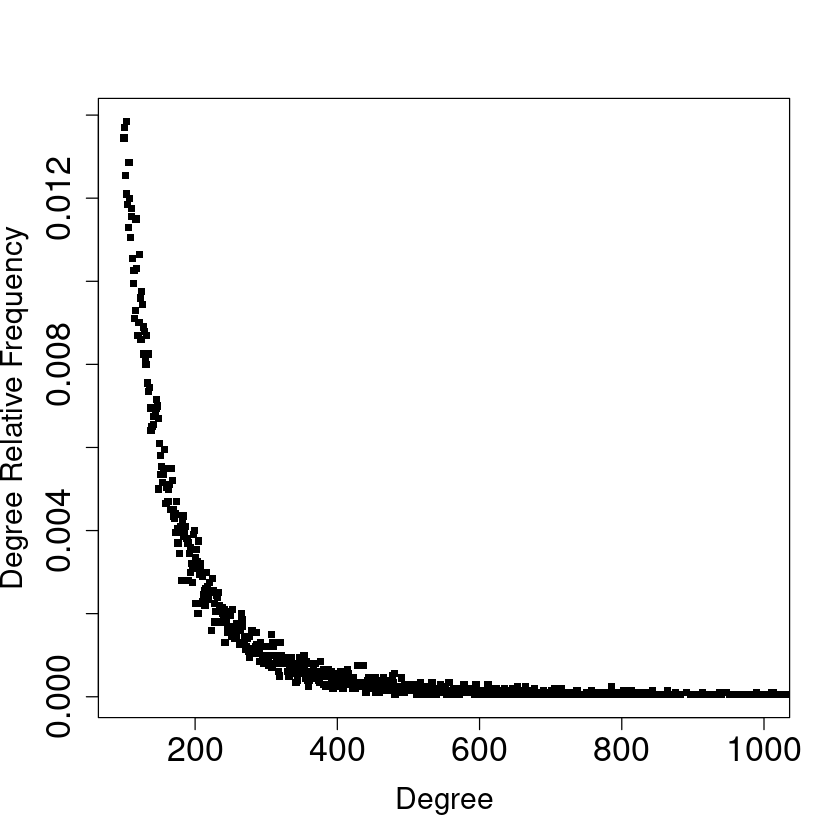}
					\subcaption{$p=0.4$}
				\end{subfigure}
				\newline
				\begin{subfigure}{.23\textwidth}
					\includegraphics[scale=0.22]{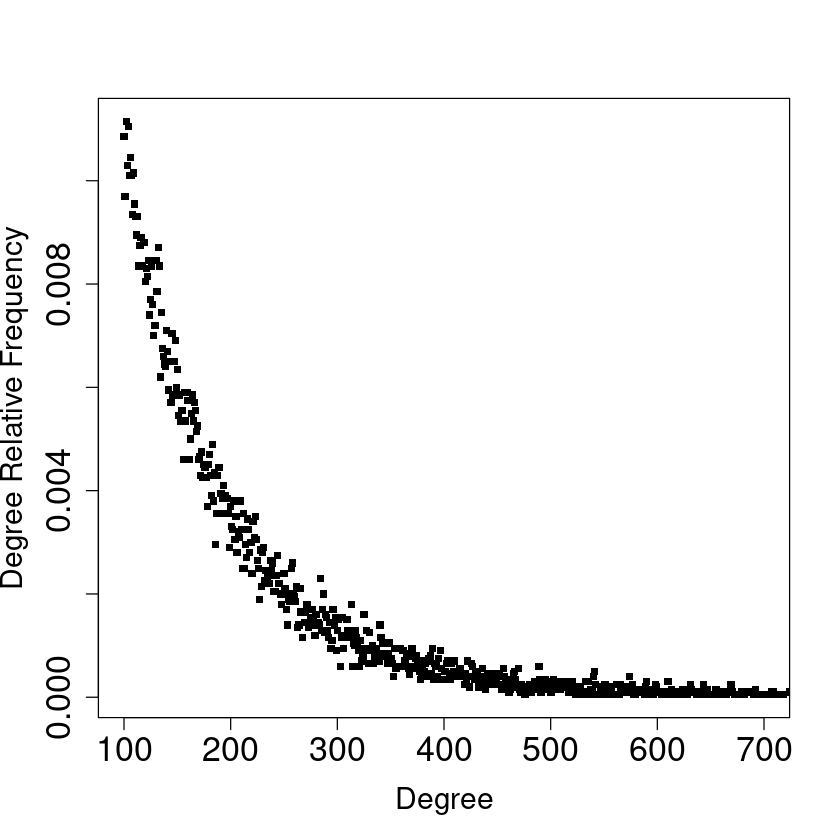}
					\subcaption{$p=0.6$}
				\end{subfigure}
				\begin{subfigure}{.23\textwidth}
					\includegraphics[scale=0.22]{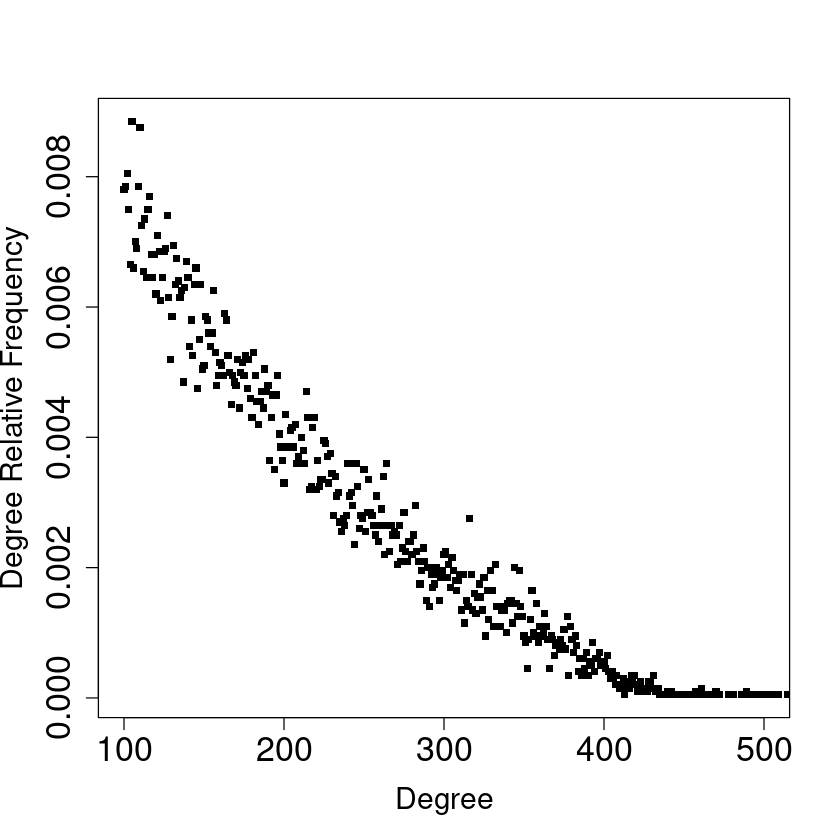}
					\subcaption{$p=0.8$}
				\end{subfigure}
				\newline
				\begin{subfigure}{.23\textwidth}
					\includegraphics[scale=0.22]{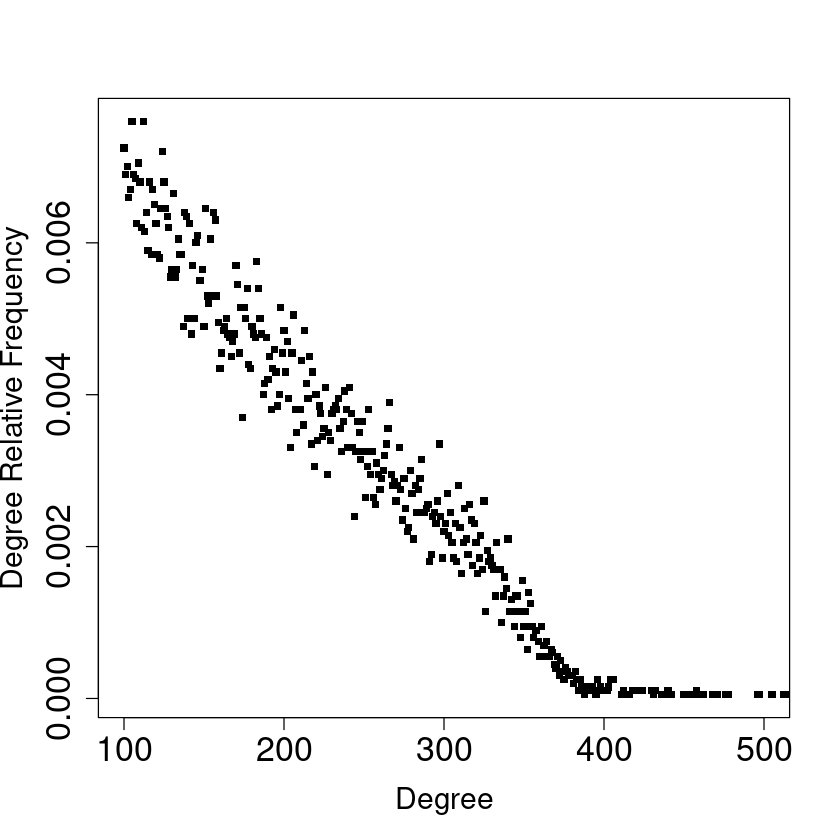}
					\subcaption{$p=0.85$}
				\end{subfigure}
				\begin{subfigure}{.23\textwidth}
					\includegraphics[scale=0.22]{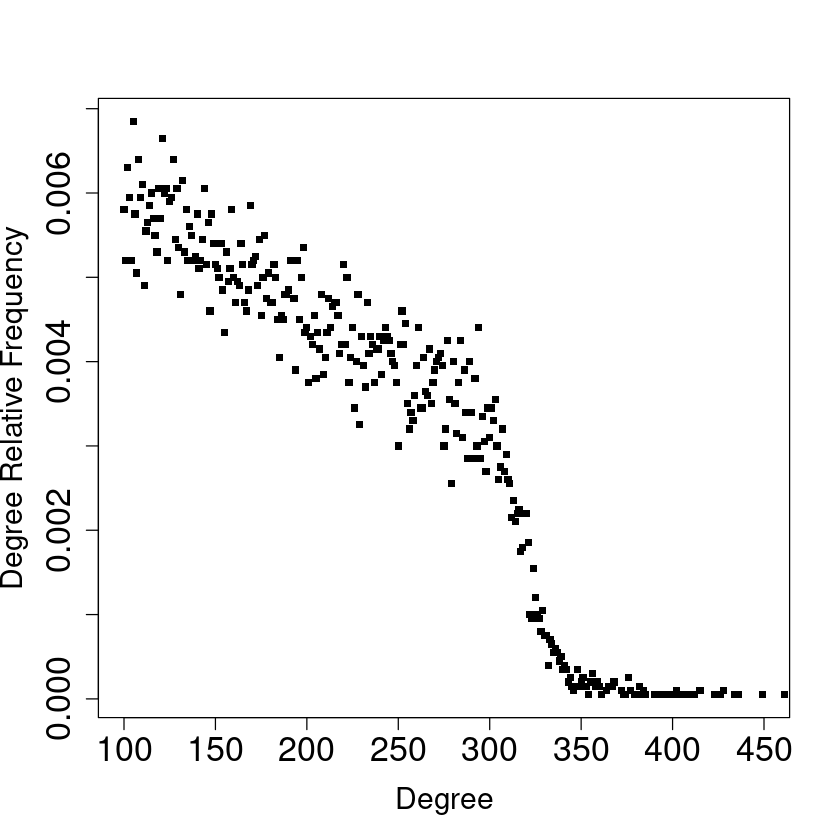}
					\subcaption{$p=0.9$}
				\end{subfigure}
				\newline
				\begin{subfigure}{.23\textwidth}
					\includegraphics[scale=0.22]{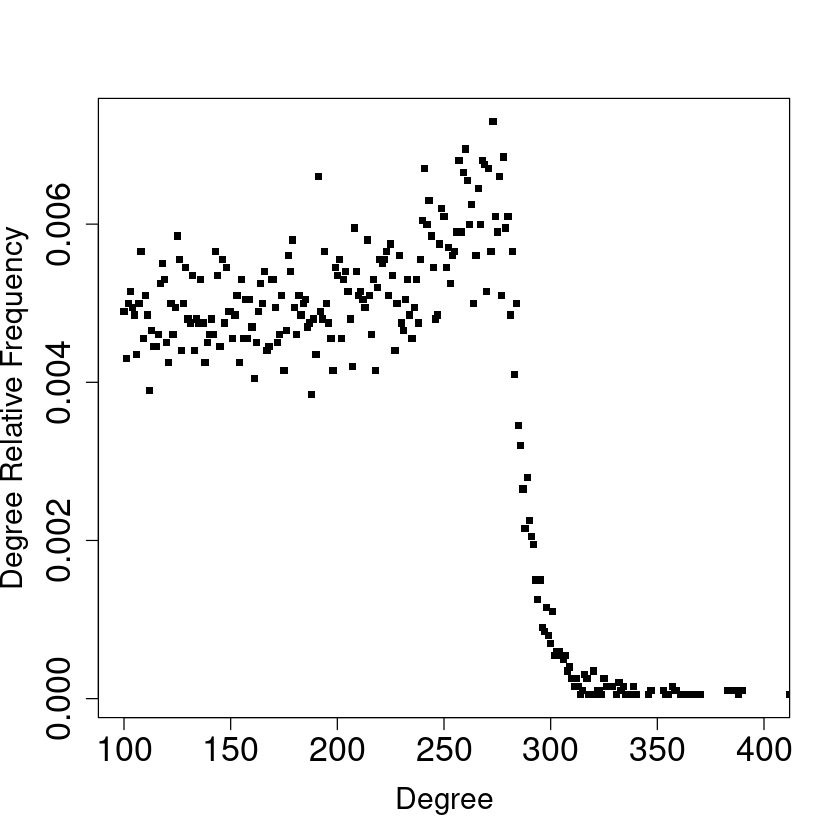}
					\subcaption{$p=0.95$}
				\end{subfigure}
				\begin{subfigure}{.23\textwidth}
					\includegraphics[scale=0.22]{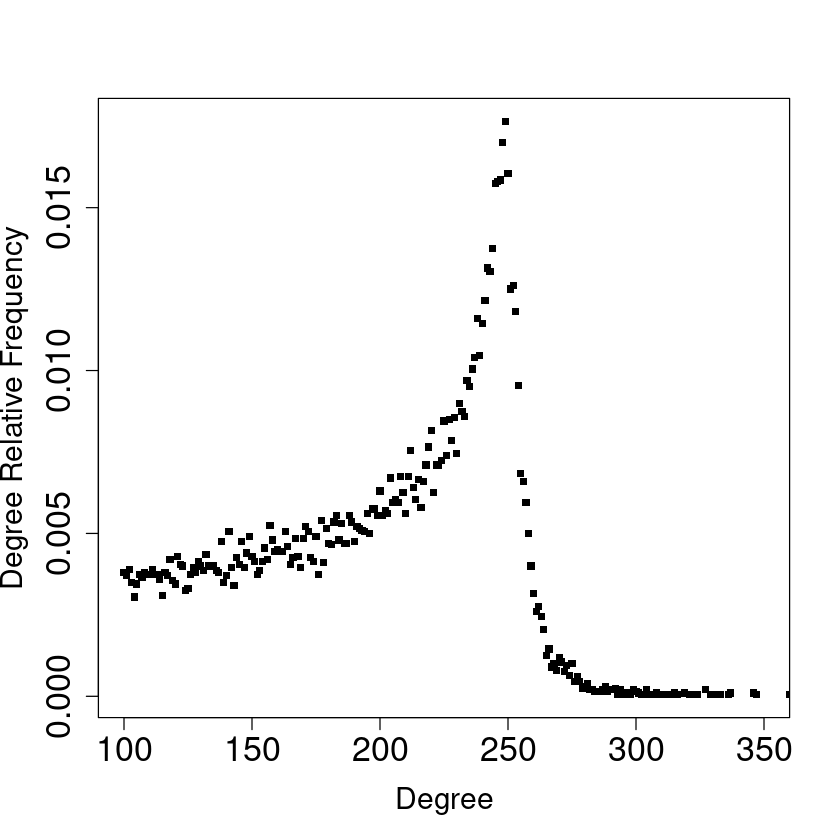}
					\subcaption{$p=1$}
				\end{subfigure}
				\caption{\label{figu}Degree distribution of the PA-APA-2 model for fixed $m$ and $t$ and different values of $p$.
					In all plots $m=100$ and $t=20000$.
					The appearance of the peak separating the two different regions is shown.}
			\end{figure}
			
			\begin{figure}
				\includegraphics[scale=0.4]{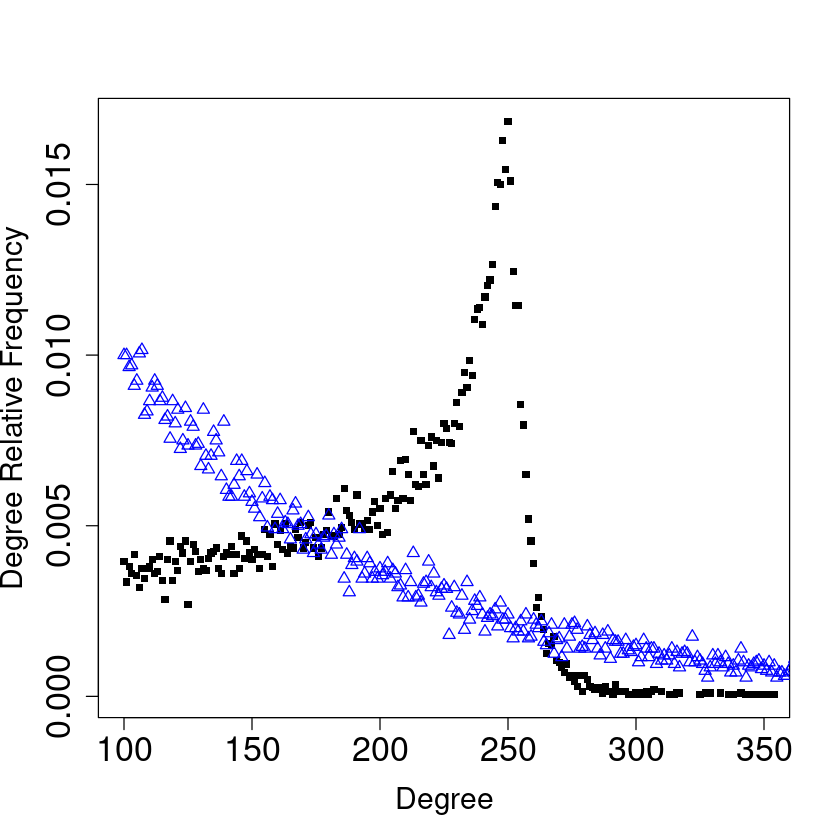}
				\caption{\label{comp}A comparison of the degree distributions of the P-APA model of Section \ref{iniz} (blue  triangles) and the PA-APA-2 model considered in Section \ref{variation} (black points). For both, $p=1$, $m=100$, $t=20000$.}
			\end{figure}

		In this section we present a simulation study of a more realistic formulation of the PA-APA model, different from that
		considered in Section \ref{iniz}, which exhibits interesting features. In the following we will call it the PA-APA-2 model. Let us define it as follows.
		Similarly to the original PA-APA model, we consider a sequence of i.i.d.\ Bernoulli random variables $(Y_{t})_{t\geq 2}$ of parameter $p \in [0,1]$ independent of the graph $(G_t)_{t\geq 1}$.
		We start the graph at time 1 with a single vertex $v_1$ with $m$ self-loops. For every $t\in \mathbb{N}^*$ we add a new vertex $v_{t+1}$ to the graph and $m$ edges between $v_{t+1}$ and the vertices already present in $G_t$.
		The $m$ vertices to which the edges have to be attached are in this case chosen with the following procedure:
		\begin{itemize}
			\item If $Y_{t+1}=0$ we select independently and with replacement $m$ random vertices $W_{t+1}^1,\dots,
				W_{t+1}^m$ from $G_t$ with probabilities
				\begin{equation*}
					\mathbb{P}(W_{t+1}^r=v_i|\mathcal{F}_t)=\frac{d(v_i,t)}{2mt}, \qquad 1\leq i\leq t,
				\end{equation*}
				for every $r\in \{1,\dots,m\}$.
			\item If $Y_{t+1}=1$ we select independently and with replacement $m$ random
				vertices $W_{t+1}^1,\dots, W_{t+1}^m$ from $G_t$ with probabilities
				\begin{equation*}
					\mathbb{P}(W_{t+1}^r=v_i|\mathcal{F}_t)=\frac{M_t+1-d(v_i,t)}{t(M_t+1-2m)}, 
				\end{equation*}
				for every $r\in \{1,\dots,m\}$, where $1\leq i\leq t$ and
				$M_t = \max_{1\leq i \leq t}d(v_i,t)$ is the maximum degree of the graph at time $t$.
		\end{itemize}
		The above probabilities differ from those in \eqref{prefa} and \eqref{aprefa} in that the anti-preferential attachment
		probabilities are calculated with respect to the current maximum degree of the graph process instead of with respect of
		the maximum theoretical degree, typically unknown to who decides the attachment rule at a certain time. This is certainly a more realistic model albeit harder to cope with.
		Here we present the results of a simulation.
		
		Figure \ref{figu} shows the actual degree distribution of the PA-APA-2 model for different values of
		the mixing parameter $p$. Notice how the degree distribution changes with $p$. The most notable effect is the appearance
		of a prominent peak separating two different regions of the distribution as $p$ increases.

		In Figure \ref{comp}, the degree distribution of the PA-APA model in the case of pure anti-preferential attachment (i.e., if $p=1$) is compared with the corresponding degree distribution (i.e. same parameters, same time $t$) of the PA-APA-2 model.
				
		In order to investigate the reasons of the presence of the peak in the degree distribution of the PA-APA-2 model we performed extensive simulations for the pure anti-preferential attachment case ($p=1$). First, we should notice (Figure \ref{dara}) that the presence of the peak is actually a transient phenomenon. To make this visible in a reasonable time we had to lower to 30 the value of the parameter $m$, which represents the number of new edges created at each time. In this case for small values of $t$ the peak is still visible while for increasing $t$ a rather slow smoothing phenomenon occurs (this can be noted in particular for $t= 1000000$). For larger values of $m$ the smoothing phenomenon still exists but it is delayed. Indeed, for instance for $m=100$, $t=1000000$, the peak is still well visible (Figure \ref{figu2}).
		
		\begin{figure}
				\begin{subfigure}{.23\textwidth}
					\includegraphics[scale=0.22]{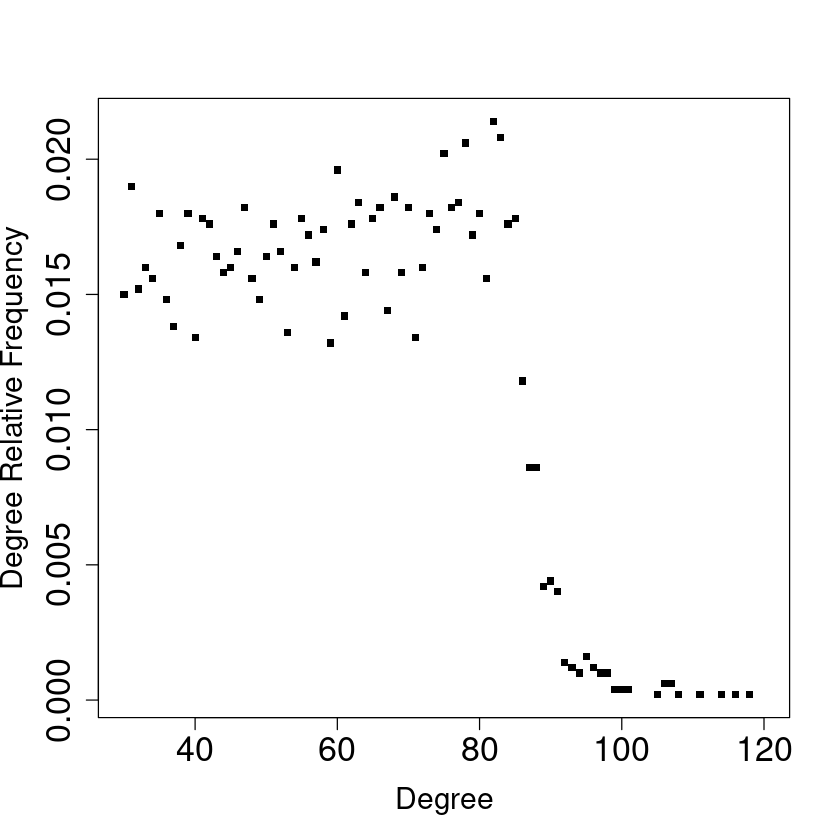}
					\subcaption{$t=5000$}
				\end{subfigure}
				\begin{subfigure}{.23\textwidth}
					\includegraphics[scale=0.22]{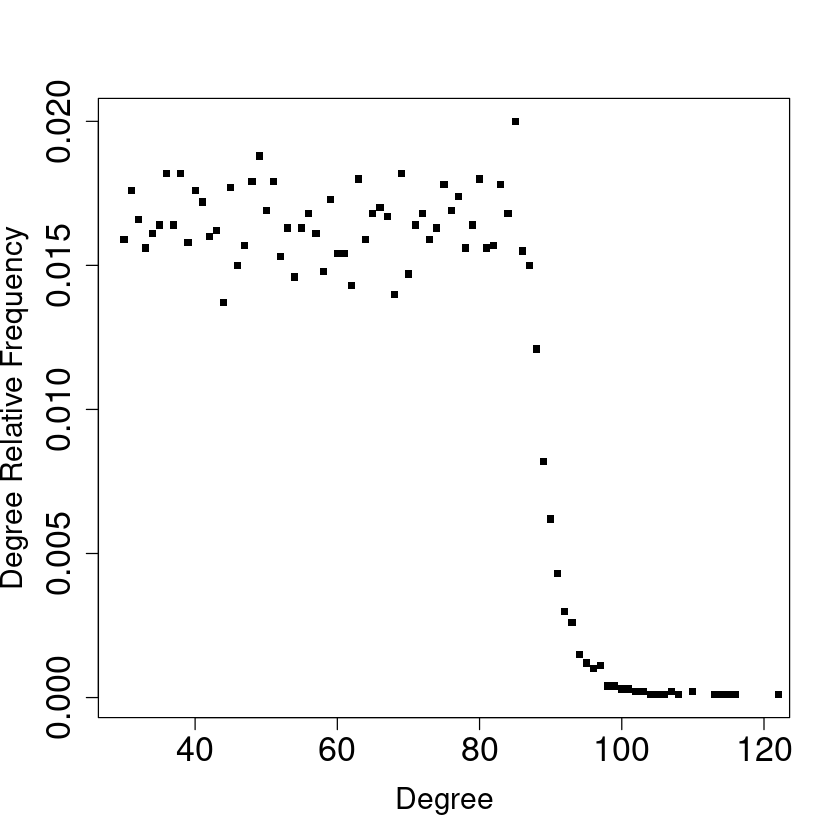}
					\subcaption{$t=10000$}
				\end{subfigure}
				\newline
				\begin{subfigure}{.23\textwidth}
					\includegraphics[scale=0.22]{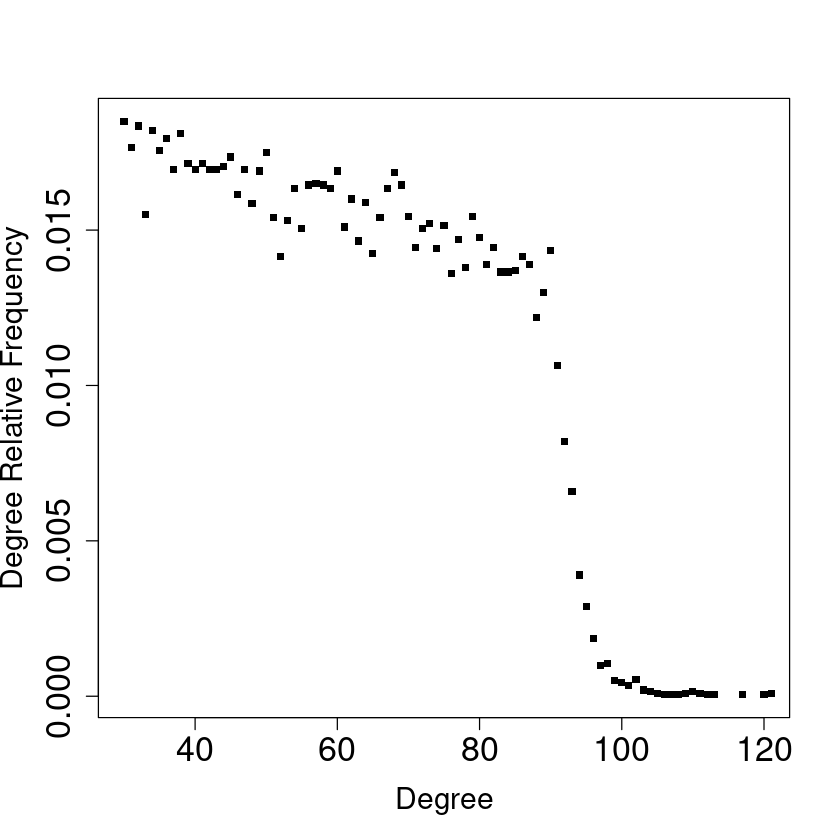}
					\subcaption{$t=20000$}
				\end{subfigure}
				\begin{subfigure}{.23\textwidth}
					\includegraphics[scale=0.22]{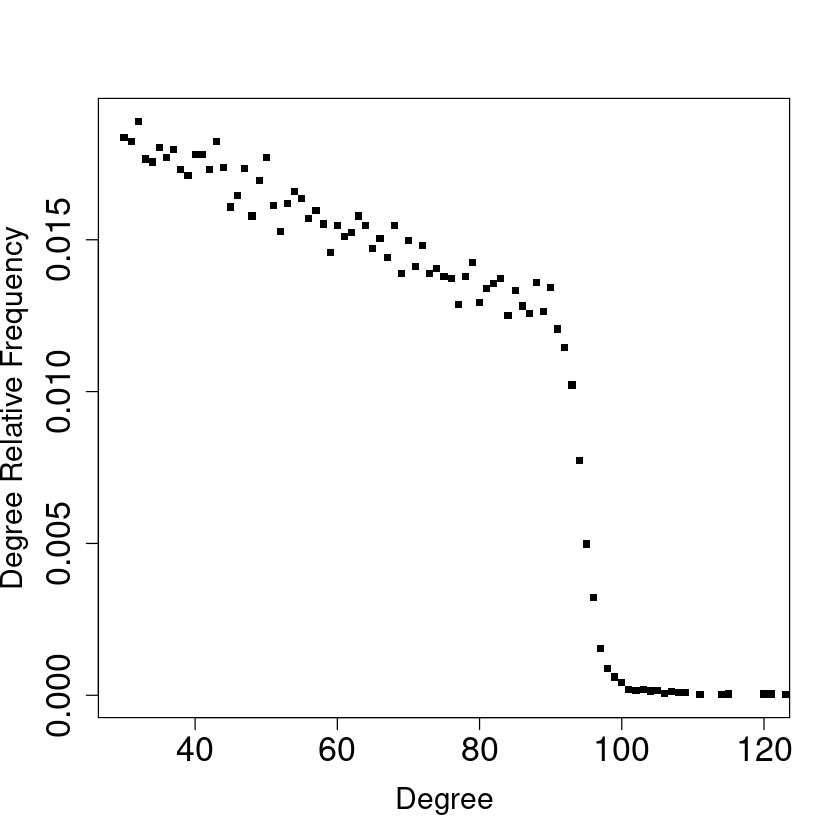}
					\subcaption{$t=50000$}
				\end{subfigure}
				\newline
				\begin{subfigure}{.23\textwidth}
					\includegraphics[scale=0.22]{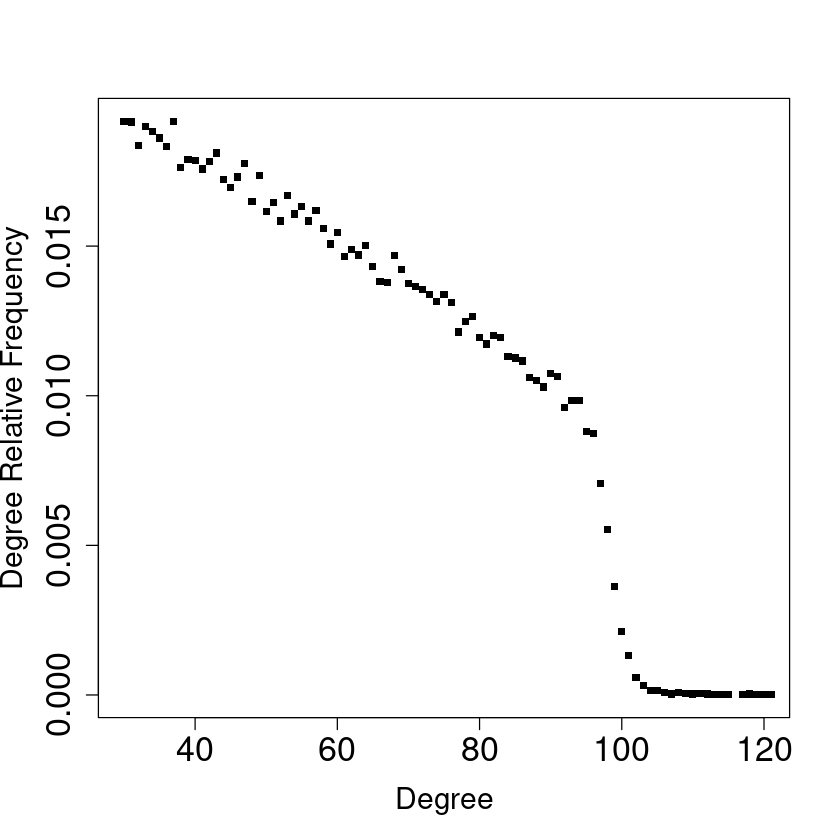}
					\subcaption{$t=100000$}
				\end{subfigure}
				\begin{subfigure}{.23\textwidth}
					\includegraphics[scale=0.22]{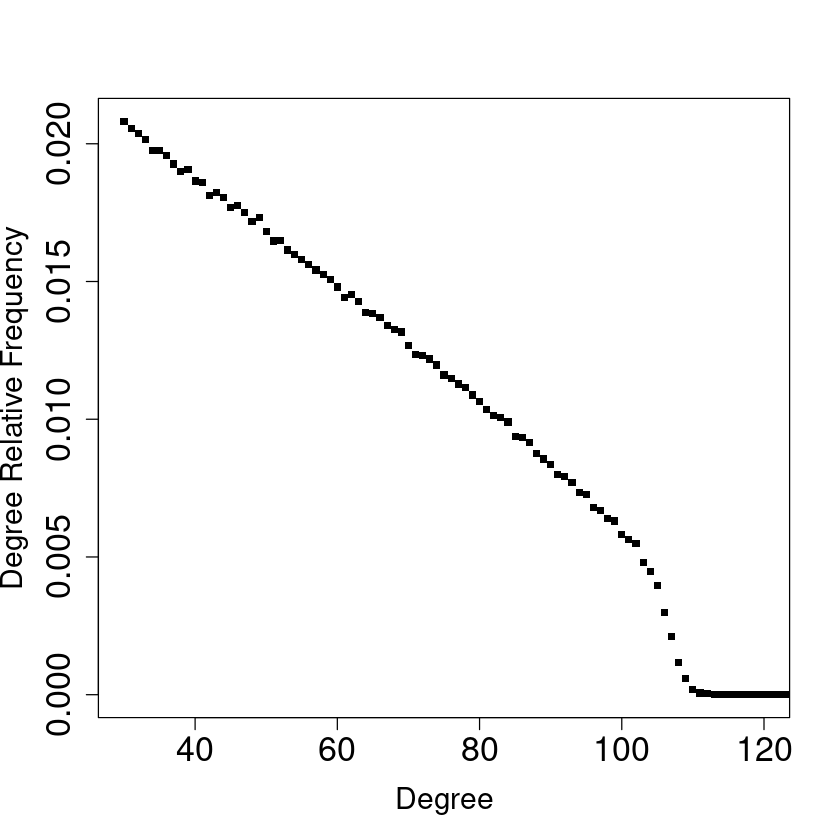}
					\subcaption{$t=1000000$}
				\end{subfigure}
				\caption{\label{dara}For increasing values of $t$, the peak present in the degree distribution of the PA-APA-2 model is smoothed out. Here we set $p=1$ and $m=30$.}
			\end{figure}
			
			The presence of the peak in the transient
			regimes of the PA-APA-2 model seems to be explained by the evolution of the related (random)
			attachment probabilities (which in turn depend on the
			evolution of the current maximum). Figure \ref{dada}
			compares the attachment probabilities of the PA-APA-2
			with those of the original PA-APA model for $p=1$, $m=30$,
			and different values of $t$. The choice of $p=1$ is justified here by the fact that
			the peak is more evident and permits a better comparison.
			Anyhow, a similar analysis can be conducted also for different values of $p$.
			The shape of the probability
			mass function is rather different in the two models. That
			of the PA-APA is more ``uniform'' leading to a spread of
			the available edges to the vertices present in the graph,
			even for small values of $t$. That of the PA-APA-2 is
			more concentrated on the recent vertices, especially
			for small values of $t$, leading thus to the appearance of the peak. For increasing values of $t$ then, the probability mass function tends to be less and less concentrated on the recent values, allowing thus a gradual smoothing.
			
			We conclude the paper by evaluating on simulated graphs the network assortativity coefficient $r$, which measures,
			as we mentioned already in the introductory section, the tendency of vertices to connect to vertices of similar degree
			and which is the Pearson correlation coefficient of the degrees at either ends of an edge selected uniformly at random (see \cite{newman2002assortative}; see also \cite{sendina2016assortativity}).
			Specifically, in the table below we simulated both the PA-APA-2 and the PA-APA
			for a given choice of the parameters ($m=100, t=20000$), for varying $p$.
			The assortativity of the graphs varies considerably from practically absence of assortativity
			($p=0$) to the pure anti-preferential attachment case for PA-APA-2 in which it has value 0.664. It seems
			reasonable that the slightly higher level of assortativity
			for the PA-APA-2 can be associated to the presence of the peak
			in the degree distribution. 
			
			\begin{table}[h]
				\centering
				\begin{tabular}{ c c c }
					 $p$ & PA-APA-2 & PA-APA \\ 
					 \hline \\
					 0 & 0.059 & 0.058 \\  
					 0.2 & 0.152 & 0.140 \\  
					 0.4 & 0.261 & 0.217 \\  
					 0.6 & 0.384 & 0.294 \\  
					 0.8 & 0.522 & 0.363 \\  
					 1 & 0.664 & 0.421    
				\end{tabular}
				\caption{Network assortativity coefficient $r$ as in \cite{newman2002assortative}, formula (3), calculated on simulated graphs with
				different values of $p$, when $m=100$, $t=20000$.}
			\end{table}

			\begin{figure}
				\includegraphics[scale=.21]{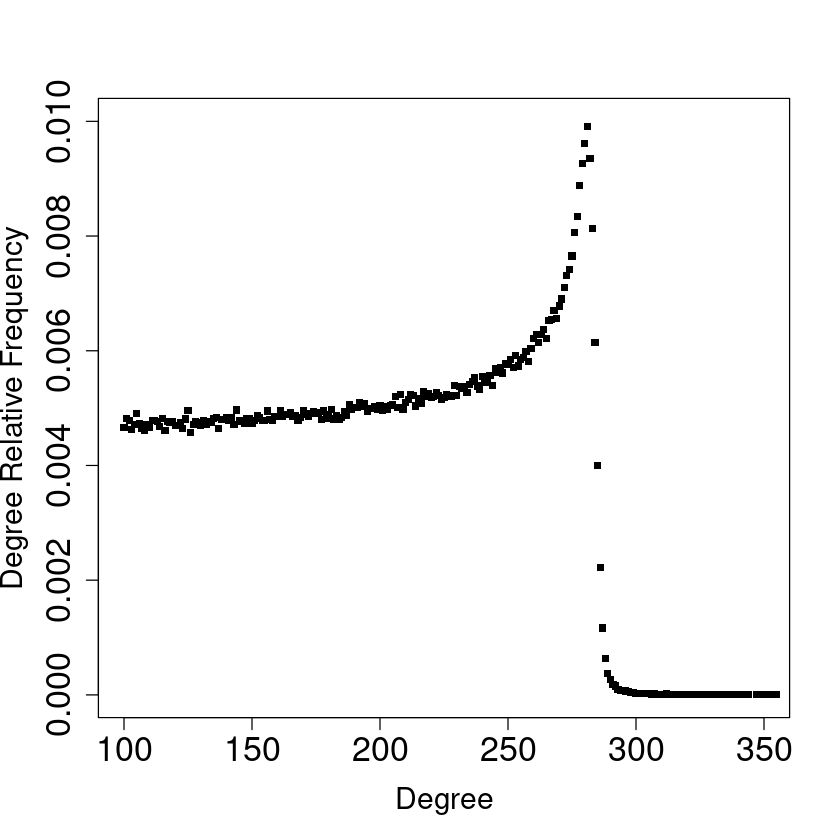}
				\includegraphics[scale=.21]{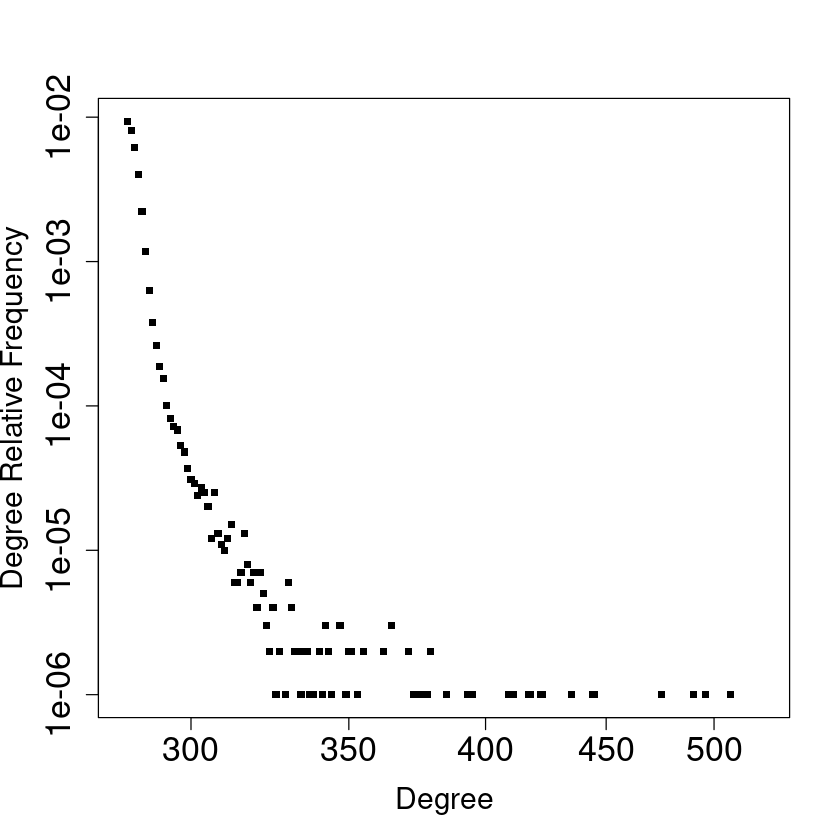}
				\caption{\label{figu2}Degree distribution of the PA-APA-2 for $p=1$ with $m=100$, $t=1000000$. In the left
					plot it is shown the complete distribution in linear scale. On the right it is shown instead a log-log
					plot restricted to the region with degree larger than the degree with maximum probability.}
			\end{figure}		
		
					\begin{figure}
				\begin{subfigure}{.23\textwidth}
					\includegraphics[scale=0.22]{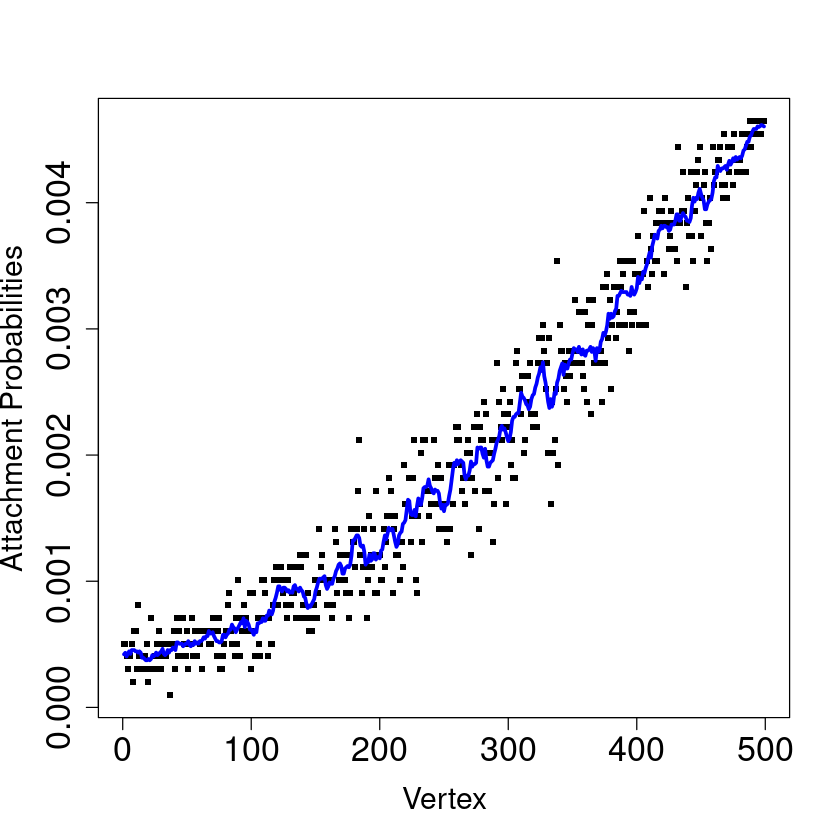}
					\subcaption{PA-APA-2, $t=500$}
				\end{subfigure}
				\begin{subfigure}{.23\textwidth}
					\includegraphics[scale=0.22]{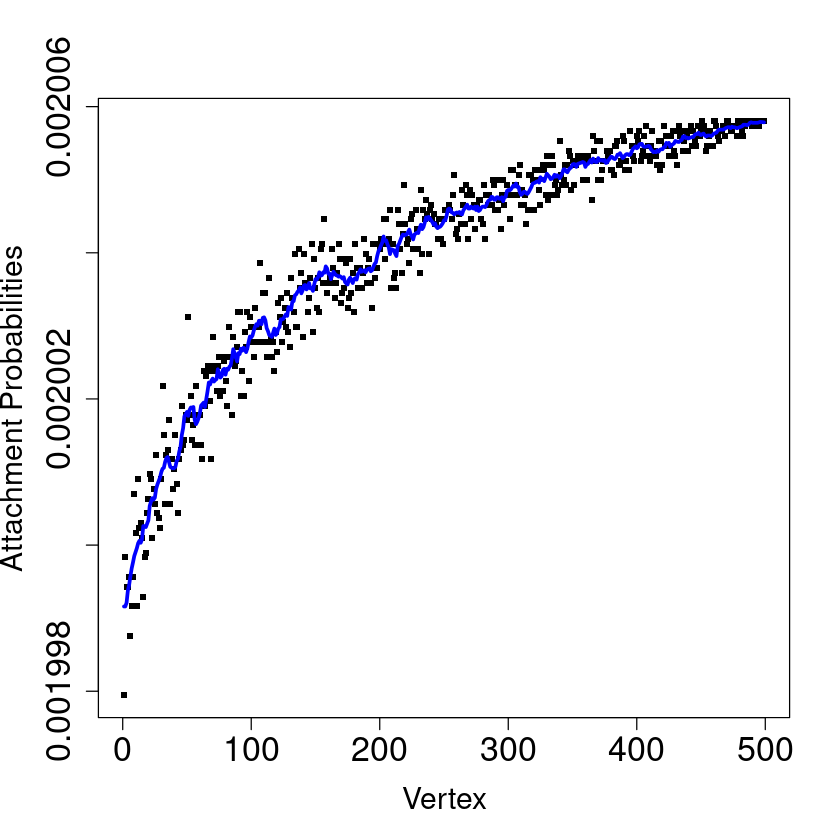}
					\subcaption{PA-APA, $t=500$}
				\end{subfigure}
				\newline
				\begin{subfigure}{.23\textwidth}
					\includegraphics[scale=0.22]{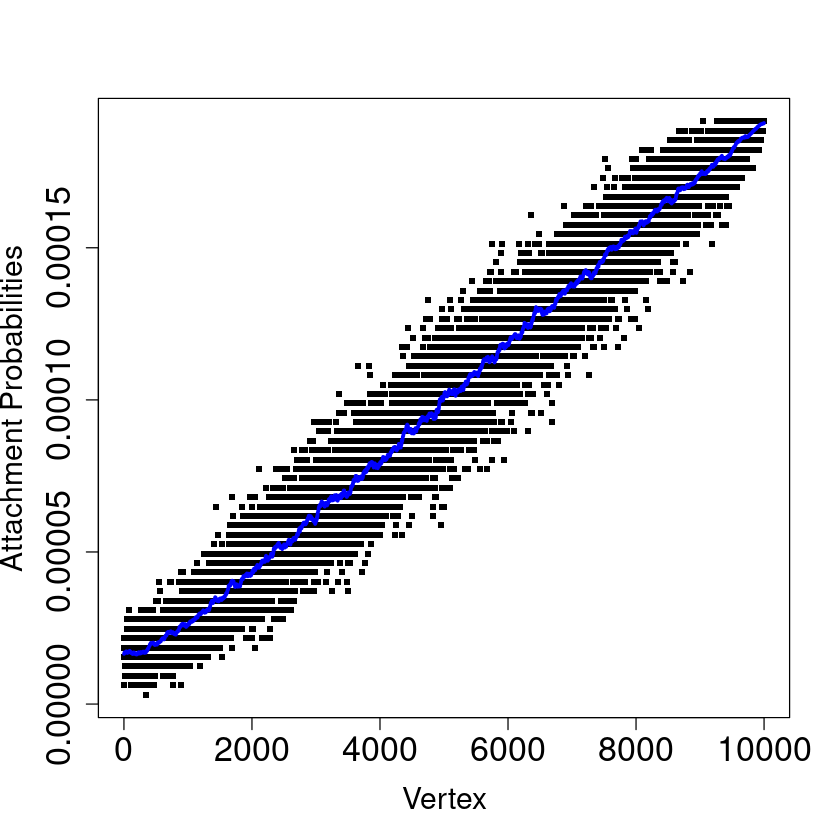}
					\subcaption{PA-APA-2, $t=10000$}
				\end{subfigure}
				\begin{subfigure}{.23\textwidth}
					\includegraphics[scale=0.22]{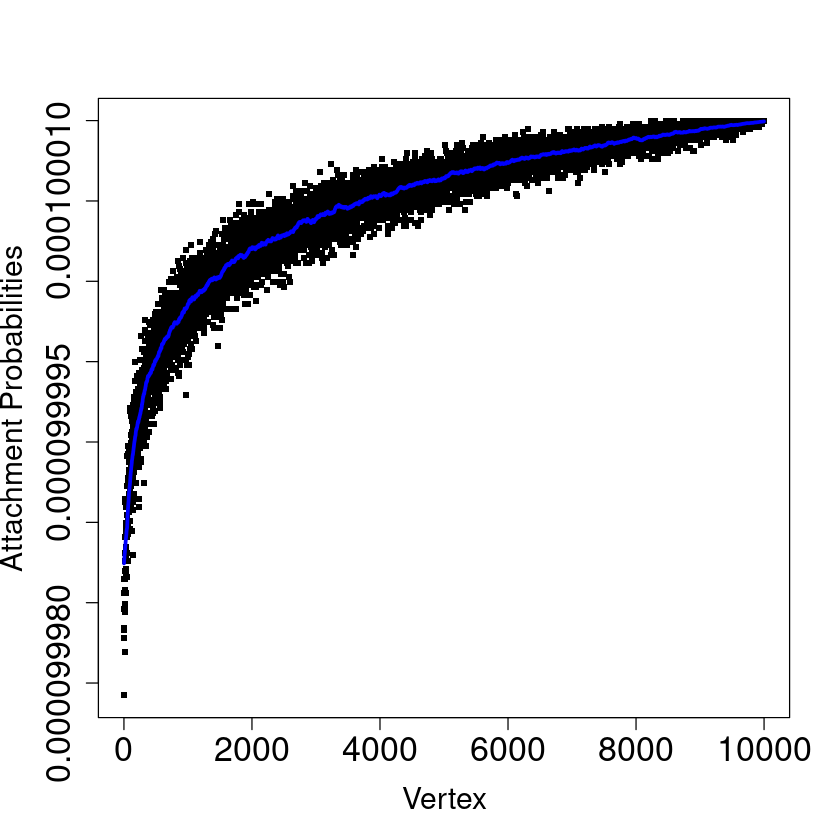}
					\subcaption{PA-APA, $t=10000$}
				\end{subfigure}
				\newline
				\begin{subfigure}{.23\textwidth}
					\includegraphics[scale=0.22]{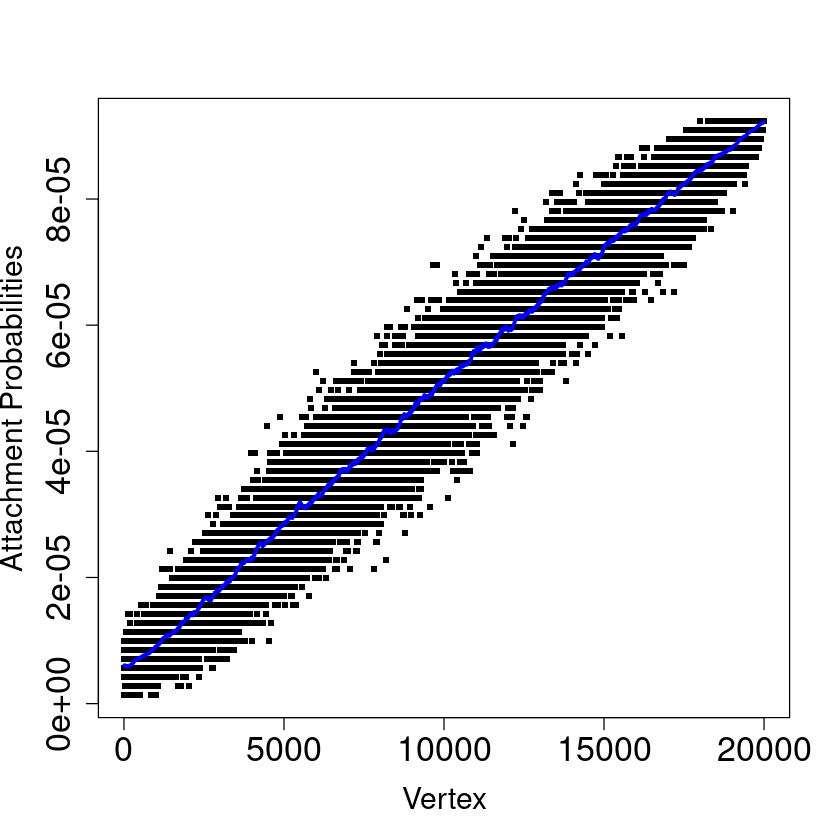}
					\subcaption{PA-APA-2, $t=20000$}
				\end{subfigure}
				\begin{subfigure}{.23\textwidth}
					\includegraphics[scale=0.22]{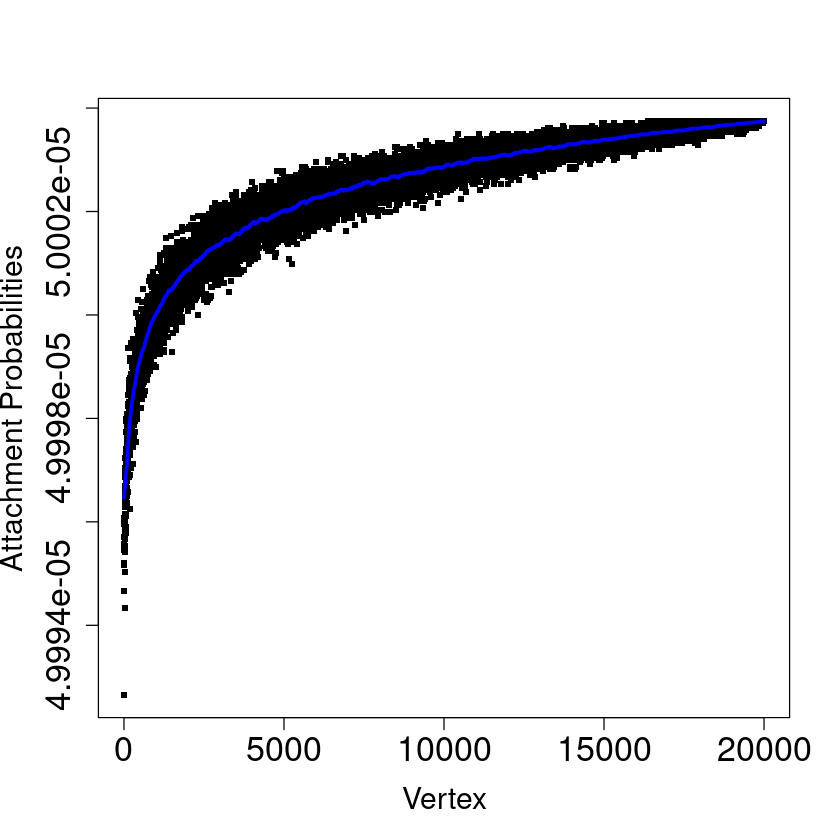}
					\subcaption{PA-APA, $t=20000$}
				\end{subfigure}
				\newline
				\begin{subfigure}{.23\textwidth}
					\includegraphics[scale=0.22]{comp3a.png}
					\subcaption{PA-APA-2, $t=100000$}
				\end{subfigure}
				\begin{subfigure}{.23\textwidth}
					\includegraphics[scale=0.22]{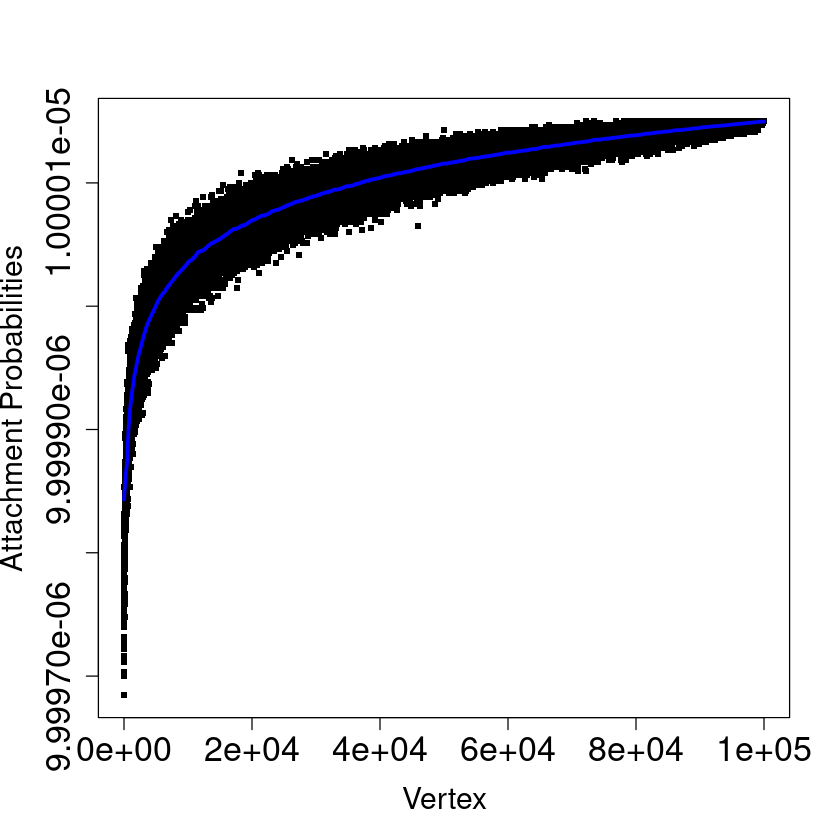}
					\subcaption{PA-APA, $t=100000$}
				\end{subfigure}
				\caption{\label{dada}Comparison of the preferential attachment probabilities of the PA-APA-2 model (left column) and the PA-APA model (right column) for different values of $t$. The models' parameters are set to $m =30$, $p=1$ (pure anti-preferential attachment).
				Note that the vertices are labelized by the time at which they appeared. In this way the $x$ axes of the above
				pictures represent all the vertices in the graph ordered
				by their appearance time: older vertices are close to the
				origin and more recent vertices close to $t$.  
				In blue, a running mean to analyze better the overall behaviour of the probabilities; the window's width is the same for both models and is equal, from top to bottom, to $10, 100, 200, 1000$.}
			\end{figure}		
		
	\section{Conclusions}
		
		The PA-APA model enriches the structure of a pure
		preferential attachment random graph by allowing edge linking
		by means of anti-preferential attachment probabilities. These
		induce in the graphs an effect of degree homogenization
		which depend on the balancing parameter $p$. It turns out that the
		parameter $p$ is linked with several aspects of the graph structure, such as
		the rate of growth of the (expected) degree and the assortativity level.
		
		Furthermore, we have compared the PA-APA model with the PA-APA-2 model, which we have analyzed only via simulations. For the latter model the anti-preferential attachment probabilities are
		calculated with respect to the random current maximum of the
		degree process. The degree homogenization is here less
		evident, mainly due to the structure of the attachment probabilities.
		A thorough theoretical analysis of the PA-APA-2 model would be interesting in principle but, based on the results we obtained
		by means of simulations, difficult in practice.

	\appendices
	
	\section{Technical lemmas and proofs of theorems}\label{app}

			Next four lemmas, for which we omit the proofs, are needed to prove Theorem \ref{propa}.

			\begin{lem}
				\label{primo}
				The limit $\lim_{t\rightarrow \infty}P(k,t)$ exists if and only if $\lim_{t\rightarrow \infty}Q(k,t)$ exists,
				in which case the two limits coincide.
			\end{lem}
			\begin{lem}
				\label{med}
				Recalling that the degree process $(d(v_i,t))_{t\ge i}$ is non decreasing for every given vertex $v_i$,
				for $m \le k \le (t-i+1)m$, $t\geq i$, and $i \ge 2$,
				\begin{align*}
					f(v_i,&k,t)=\sum_{j=1}^{(k-m) \wedge m}\binom{m}{j}H(k-j,t-1)^j \notag \\
					\times{}& K(k-j,t-1)^{m-j}\mathbb{P}(d(v_i,t-1)=k-j).
				\end{align*}
			\end{lem}
			\begin{lem}
				\label{ult}
				for $m \le k \le (t-i+1)m$, $t\geq i$, $i \ge 2$,
				\begin{equation*}
					\mathbb{P}(d(v_i,t)=k)=\sum_{s=i}^{t}f(v_i,k,s)\prod_{j=s}^{t-1}K(k,j+1)^m.
				\end{equation*}
			\end{lem}
			\begin{lem}[Stolz--Cesaro's theorem]
				\label{propprop}
				Let $(x_n)_{n\in \mathbb{N}^*}$ and $(y_n)_{n\in \mathbb{N}^*}$ be two sequences of real  numbers. Suppose that $y_n>0$ and $y_n<y_{n+1}$ for every $n \in \mathbb{N}^*$. Further, assume that $y_n\rightarrow \infty$ as $n\rightarrow \infty$. If the following limit exists:
				\begin{equation*}
					\lim_{n\rightarrow \infty}\frac{x_{n+1}-x_n}{y_{n+1}-y_n} = l,
				\end{equation*}
				then also the limit $\lim_{n \rightarrow \infty}x_n/y_n$ exists and is equal to $l$.
			\end{lem}
			
						\begin{proof}[\textbf{\textup{Proof of Theorem \ref{propa}}}]
							We adapt a proof technique from \cite{MR2807479} to determine the asymptotics of the expected proportion of vertices of a given degree.						
						
							Let $p\in [0,1]$. We start by proving that $P(m)$ exists and is equal to $2/(2+m+mp)$
							if $p \in [0,1)$ and to $1/(m+1)$ if $p=1$. Then we extend the result to all $k> m$.
				
						 	For $k=m$, note that $\mathbb{P}(d(v_t,t)=m)=1$, $t \ge 2$, as $v_t$ is connected to $m$ vertices when it is added to $G_{t-1}$. Moreover, $\mathbb{P}(d(v_i,t)=m)=\mathbb{P}(d(v_i,t-1)=m)K(m,t)^m$, $t \ge 3$, as a vertex with degree $m$ at time $t-1$ keeps its degree unchanged at time $t$ provided it receives none of the $m$ edges attached to the newborn vertex $v_t$. Now observing that, for $t\geq 3$,
							\begin{align*}
								Q(m,t) & =\frac{1}{t-1}\sum_{j=2}^{t-1}\mathbb{P}(d(v_j,t-1)=m)K(m,t)^m\notag\\
								&+\frac{1}{t-1},
							 \end{align*}
							 we obtain the following recurrence equation for $Q(m,t)$:
						 	\begin{align}
							 	\label{equaz}
							 	Q(m,t)&=\frac{t-2}{t-1}K(m,t)^mQ(m,t-1)\notag \\
							 	&+\frac{1}{t-1} \hspace{0.3cm}, \qquad  t\geq 3.
						 	\end{align}
						 	Iterating on $t$ it is immediate to show that the solution of \eqref{equaz} is
						 	\begin{align*}
							 	Q(m,t)&=\frac{1}{t-1}\prod_{j=2}^{t-1}K(m,j+1)^m\notag \\
							 	& \times \left[1+\sum_{h=2}^{t-1}\prod_{k=2}^{h}K(m,k+1)^{-m}\right], \quad t\geq3.
						 	\end{align*}
						 	Next we define, for $t\geq 3$, the following numerical sequences:
						 	\begin{align*}
						 		&x_t=1+\sum_{h=2}^{t-1}\prod_{k=2}^{h}K(m,k+1)^{-m}, \\
								&y_t=(t-1)\prod_{j=2}^{t-1}K(m,j+1)^{-m},
							\end{align*}
						 	so that $Q(m,t)=x_t/y_t$. Observe that $(y_t)_{t\geq 3}$ is a positive strictly increasing sequence with $y_t\rightarrow \infty$ as $t\rightarrow \infty$. Moreover,
						 	\begin{equation*}
							 	\frac{x_{t+1}-x_t}{y_{t+1}-y_t}=\frac{1}{t-(t-1)K(m,t+1)^m}, \qquad t\geq 3.
							\end{equation*}
							We substitute the explicit expression for function $K$ and rewrite the above ratio
							as 
						 	\begin{align*}
							 	\frac{x_{t+1}-x_t}{y_{t+1}-y_t}&=\left(mp\frac{2m+1/t-m/t}{2m+1/t-2m/t} \right.\\
							 	&\left.+(1-p)
							 	\frac{m}{2}+1+o(1/t)\right)^{-1}.
							\end{align*}
							Now, Lemma \ref{propprop} and Lemma \ref{primo} immediately yield
							\begin{align*}
							 	\lim_{t\rightarrow \infty}P(m,t)
							 	&= \lim_{t\rightarrow \infty}Q(m,t)\\ &= \lim_{t\rightarrow \infty}\frac{x_{t+1}-x_t}{y_{t+1}-y_t}
							 	=\frac{2}{2+m+mp},
							\end{align*}
							which reduces to $\lim_{t\rightarrow \infty}P(m,t) = 1/(1+m)$ if $p=1$. 
			 	
						 	For $k > m$, $t \ge 2$, we proceed as in the following.
						 	First, note that $\mathbb{P}(d(v_1,t)=k) = 0$ if $k < 2m$ or $k > (t+1)m$, and for $i \ge 2$, $\mathbb{P}(d(v_i,t)=k) = 0$ if $k >(t-i+1)m$. Observe that
						 	\begin{equation*}
							 	P(k,t)=\frac{1}{t}\sum_{j=2}^{t}\mathbb{P}(d(v_j,t)=k)+\frac{1}{t}\mathbb{P}(d(v_1,t)=k),
							\end{equation*}
							which, by using Lemmas \ref{med} and \ref{ult}, can be written as
						 	\begin{align*}
							 	P&(k,t)
							 	=\frac{1}{t}\sum_{j=2}^{t}\Biggl[\sum_{s=j}^{t}\Biggl(\sum_{l=1}^{(k-m)\wedge m}G(l,s) \\
							 	& \times \mathbb{P}(d(v_j,s-1)=k-l)\Biggr)\prod_{h=s}^{t-1}K(k,h+1)^m\Biggr] \\
							 	& +\frac{1}{t}\mathbb{P}(d(v_1,t)=k),
							\end{align*}
						 	where $G(l,s)=\binom{m}{l}H(k-l,s)^lK(k-l,s)^{m-l}$.
						 	By exchanging the order of summations we obtain
						 	\begin{align*}
							 	P&(k,t) =\sum_{l=1}^{(k-m)\wedge m}\frac{1}{t}\sum_{j=2}^{t}\Biggl[\sum_{s=j}^{t}G(l,s)\\
							 	&\times\mathbb{P}(d(v_j,s-1)=k-l)\left(\prod_{h=s}^{t-1}K(k,h+1)^m\right)\Biggr]\\
							 	&+\frac{1}{t}\mathbb{P}(d(v_1,t)=k) \\
								& =\sum_{l=1}^{(k-m)\wedge m} \frac{1}{t}\sum_{s=2}^{t}\Biggl[G(l,s)\prod_{h=s}^{t-1}K(k,h+1)^m \\
								&\times \sum_{j=2}^{s}\mathbb{P}(d(v_j,s-1)=k-l)\Biggr]\\
								&+\frac{1}{t}\mathbb{P}(d(v_1,t)=k). \notag
							\end{align*}
							Recalling that $\mathbb{P}(d(v_s,s-1)\ne 0)=0$ and since
						 	$\sum_{j=2}^{s-1}\mathbb{P}(d(v_j,s-1)=k-l)=(s-2)Q(k-l,s-1)$,
						 	we have that
						 	\begin{align*}
							 	P&(k,t) =  \sum_{l=1}^{(k-m)\wedge m}\left(\frac{1}{t}\prod_{h=1}^{t-1}K(k,h+1)^m \right)\\
							 	&\times \sum_{s=2}^{t}G(l,s)(s-2)Q(k-l,s-1) \\
							 	&\times \prod_{r=1}^{s-1}K(k,r+1)^{-m}
							 	+\frac{1}{t}\mathbb{P}(d(v_1,t)=k). \notag
							\end{align*}
						 	Wishing to apply Lemma \ref{propprop}, we define the following numerical sequences $(z_{t,l})_{t\ge 2}$, for each $l \ge 1$, and $(w_t)_{t\ge 2}$:
						 	\begin{align*}
						 		z_{t,l} & =\sum_{s=2}^{t}G(l,s)(s-2)Q(k-l,s-1) \\
						 		& \times \prod_{r=1}^{s-1}K(k,r+1)^{-m}, \qquad l \ge 1,\\
						 		w_t & = t\prod_{h=1}^{t-1}K(k,h+1)^{-m}.
						 	\end{align*}
						 	Notice that $(w_t)_t$ is strictly positive and strictly increasing towards infinity. Moreover, by Lemma \ref{propprop},
							\begin{align*}
								P&(k) =\sum_{l=1}^{(k-m)\wedge m} \lim_{t\rightarrow \infty}\frac{z_{t+1,l}-z_t}{w_{t+1}-w_t}\\
								&=\sum_{l=1}^{(k-m)\wedge m} \lim_{t\rightarrow \infty}\frac{G(l,t+1)(t-1)}{t+1-tK(k,t+1)^m}Q(k-l,t) \\
								& =\sum_{l=1}^{(k-m)\wedge m} \lim_{t\rightarrow \infty} L_{k,l}(t) Q(k-l,t). \notag
							\end{align*}
						 	Taking into account Lemma \ref{primo}, we derive the limit value of each term of the latter sum. Specifically,
						 	for $l=1$ we have
						 	\begin{align*}
							 	&L_{k,1}(t)\\
							 	&= 	\frac{mp\frac{2m+1/t-k/t+1/t}{2m+1/t-2m/t}+m(1-p)
							 	\frac{k-1}{2m}+o(1/t)}{1+mp\frac{2m+1/t-k/t}{2m+1/t-2m/t}+m(1-p)\frac{k}{2m}+o(1/t)} \\
							 	&\overset{t\rightarrow \infty}{\longrightarrow}
								\frac{2mp+(1-p)(k-1)}{2+2mp+(1-p)k}.
							\end{align*}
							Similarly, it is straightforward to see that $\lim_{t \to \infty}L_{k,l}(t)=0$ for each $l \ge 2$.
							
						 	Hence, for $k > m$,
						 	\begin{equation}
							 	\label{zanguni}
							 	P(k)=\frac{2mp+(1-p)(k-1)}{2+2mp+(1-p)k}P(k-1).
						 	\end{equation}
						 	If $p=1$, by iterating backwards $(k-m)$ times, we obtain
						 	\begin{align*}
							 	P(k)&=P(m)\left(\frac{m}{m+1}\right)^{k-m}\\
							 	&=\frac{1}{m+1}\left(\frac{m}{m+1}\right)^{k-m},
							\end{align*}
						 	proving the second part of the statement.
						 	If $p\in [0,1)$, rewriting \eqref{zanguni} as
							\begin{equation*}
								P(k)=\frac{k+\frac{p(2m+1)-1}{1-p}}{k+\frac{2(mp+1)}{1-p}}P(k-1),
							\end{equation*}
						 	and again iterating backward $(k-m)$ times we get
						 	\begin{align*}
							 	P(k)& =
								P(m)\frac{\Gamma\left(m+1+\frac{2(m+1)}{1-p}\right)}{\Gamma\left(m+1+\frac{p(2m+1)-1}{1-p}\right)}\\
								&\times\frac{\Gamma\left(k+1+\frac{p(2m+1)-1}{1-p}\right)}{\Gamma\left(k+1+\frac{2(m+1)}{1-p}\right)} \\
								& = \frac{2}{2+m+mp}\cdot \frac{\xi(k)}{\xi(m)}, \notag
							\end{align*}
							which proves the first part of the statement.
						\end{proof}
						
									\begin{proof}[\textbf{\textup{Proof of Theorem \ref{waif}}}]
										Let us proceed by induction.
										We will only prove the result for $i\geq 2$ as the case $i=1$ can be proved easily.
										Note first that for $t=i$ we have $\mathbb{E}[d(v_i,i)] = (1+\delta_{i1})m$. Let us now suppose that \eqref{wwaif} holds for some $t>i$. Since the increment $d(v_i,t+1)-d(v_i,t)$, conditional on $\mathcal{F}_t$, follows a binomial distribution with parameters $m$ and
										$H(d(v_i,t),t+1)$
										we have
										\begin{align}
											\label{expe}
											&\mathbb{E}[d(v_i,t+1)|\mathcal{F}_t]\notag \\
											& = d(v_i,t)+\mathbb{E}[d(v_i,t+1)-d(v_i,t)|\mathcal{F}_t] \notag\\
											& =d(v_i,t)\left(1+(1-p)\frac{1}{2t}-pe_t\right)+p\,m\,c_t.
						 				\end{align}
										Taking expectations on both sides and using the inductive hypothesis we obtain
										\begin{align*}
											\mathbb{E}[d(v_i,t+1)]
											&= \left[m\prod_{l=i}^{t-1}C(l,p,m)\right.\\
											&+ \left.\sum_{l=i}^{t-1}m\,p\,c_l
											\prod_{h=l+1}^{t-1}C(h,p,m)\right]\\
											& \times \left(1+(1-p)\frac{1}{2t}-pe_t\right)\\
											&+p\,m\,c_t m\prod_{l=i}^{t}C(l,p,m)\\
											&+\sum_{l=i}^{t}m\,p\,c_l\prod_{h=l+1}^{t}C(h,p,m),\notag
										\end{align*}
										as required.	
									\end{proof}
									
							\begin{proof}[\textbf{\textup{Proof of Theorem \ref{sqrtno}}}]
								Observe that for every $t\geq i+1$ we have
								\begin{align}
									\label{calif}
									& \left|\mathbb{E}\left[\frac{d(v_i,t)}{\sum_{l=i}^{t-1}\ln\left(1+c_l\right)\prod_{h=l+1}^{t-1}\left(1-e_h\right)}\right]-m\right|\notag \\
									&\leq \left|\frac{m\prod_{l=i}^{t-1}\left(1-e_l\right)}{\sum_{l=i}^{t-1}\ln\left(1+c_l\right)\prod_{h=l+1}^{t-1}\left(1-e_h\right)}\right| \notag\\
									&  +\left|\frac{\sum_{l=i}^{t-1}m\,c_l\prod_{h=l+1}^{t-1}\left(1-e_h\right)}{\sum_{l=i}^{t-1}\ln\left(1+c_l\right)\prod_{h=l+1}^{t-1}\left(1-e_h\right)}-m\right|.
								\end{align}
								It is not difficult to see that the first term of the right-hand side of \eqref{calif} vanishes as $t \to \infty$.
								Indeed,
								\begin{align*}
									0 &\le \frac{m\prod_{l=i}^{t-1}\left(1-e_l\right)}{\sum_{l=i}^{t-1}\ln\left(1+c_l\right)\prod_{h=l+1}^{t-1}\left(1-e_h\right)}\\
									& \le
									\frac{m}{\sum_{l=i}^{t-1}\ln\left(1+c_l\right)}
									\le \frac{m}{\sum_{l=i}^{t-1}\ln\left(1+1/l\right)},
								\end{align*}
								which goes to zero for $t\to \infty$ as $\sum_{l=i}^{t-1}\ln\left(1+1/l\right)=\ln t-\sum_{l=1}^{i-1}\ln\left(1+1/l\right)$.
								In order to prove the limit \eqref{cali} there remains to show that the second term in the right-hand side of \eqref{calif}
								vanishes as well.
								To see this, first observe that the infinite product $a=\prod_{l=i}^{\infty}\left(1-e_l\right)$
								is convergent as the related series $\sum_{l=i}^\infty e_l$ converges by Raabe's test. Moreover, $0<a<1$ as $0<(1-e_l)<1$ for each $l > 1$.
								Consequently,
								\begin{align}
									\label{rab}
									&\lim_{t\rightarrow \infty}\sum_{l=i}^{t-1}\ln\left(1+c_l\right)\prod_{h=l+1}^{t-1}\left(1-e_h\right)\notag\\
									&\geq \lim_{t\rightarrow \infty}\prod_{h=i}^{t-1}\left(1-e_h\right)\left[\ln t-\sum_{l=1}^{i-1}\ln\left(1+\frac{1}{l}\right)\right]=\infty
								\end{align}
								Now
								\begin{align*}
									&\left|\frac{\sum_{l=i}^{t-1}mc_l\prod_{h=l+1}^{t-1}\left(1-e_h\right)}{\sum_{l=i}^{t-1}\ln\left(1+c_l\right)\prod_{h=l+1}^{t-1}\left(1-e_h\right)}-m\right|\\
									&=\left|\frac{m\sum_{l=i}^{t-1}\left[\prod_{h=l+1}^{t-1}\left(1-e_h\right)\right]\left(c_l-\ln\left(1+c_l\right)\right)}{\sum_{l=i}^{t-1}\ln\left(1+c_l\right)\prod_{h=l+1}^{t-1}\left(1-e_h\right)}\right|.
								\end{align*}
								Recall that for $x\in (0,1]$, expanding the logarithm around zero, we can write
								\begin{equation*}
									x-\ln(1+x)=\frac{x^2}{2(1+\eta)^2}
								\end{equation*}
								where $\eta \in (0,x)$. Noticing that, for $l\geq i$, we have $c_l< 1$, and we obtain
								$0<c_l-\ln\left(1+c_l\right)<c_l^2/2$. Hence
								\begin{align*}
									&\left|\frac{m\sum_{l=i}^{t-1}\left[\prod_{h=l+1}^{t-1}\left(1-e_h\right)\right]\left(c_l-\ln\left(1+c_l\right)\right)}{\sum_{l=i}^{t-1}\ln\left(1+c_l\right)\prod_{h=l+1}^{t-1}\left(1-e_h\right)}\right|\\
									& < \frac{(m/2)\sum_{l=i}^{t-1}\left[\prod_{h=l+1}^{t-1}\left(1-e_h\right)\right]c_l^2}{\sum_{l=i}^{t-1}\ln\left(1+c_l\right)\prod_{h=l+1}^{t-1}\left(1-e_h\right)}\\
									& \leq \frac{(m/2)\sum_{l=i}^{t-1}c_l^2}{\sum_{l=i}^{t-1}\ln\left(1+c_l\right)\prod_{h=l+1}^{t-1}\left(1-e_h\right)}.\notag
								\end{align*}
								Finally, recalling \eqref{rab} and since $\sum_{l=i}^{\infty}c_l^2<\infty$, it follows that
								\begin{equation*}
									\lim_{t \rightarrow \infty}\left|\frac{\sum_{l=i}^{t-1}mc_l\prod_{h=l+1}^{t-1}\left(1-e_h\right)}{\sum_{l=i}^{t-1}\ln\left(1+c_l\right)\prod_{h=l+1}^{t-1}\left(1-e_h\right)}-m\right|=0.
								\end{equation*}
							\end{proof}
							
					\begin{proof}[\textbf{\textup{Proof of Theorem \ref{esta}}}]
						First, note that
						\begin{align*}
							&\prod_{l=i}^{t-1}C(l,p,m)
							\leq \exp\left(\sum_{l=i}^{t-1}\ln\left(1+\frac{1-p}{2l}\right)\right) \\
							&\leq
							t^{\frac{1-p}{2}}\exp\left(-\frac{1-p}{2}\ln t+\frac{1-p}{2}\sum_{l=1}^{t-1}\frac{1}{l}\right)  \\
							&=t^{\frac{1-p}{2}}\exp\left(\frac{1-p}{2}\left(\sum_{l=1}^{t-1}\frac{1}{l}-\ln t\right)\right)\\
							&=O\left(t^{(1-p)/2}\right), \notag
						\end{align*}			
						Now, let $l_0$ be such that $l>(1-p)^{-1}-(2m)^{-1}$ for all $l\geq l_0$. For $t\geq \max(l_0+1,i+1)$,
						\begin{align*}
							&\mathbb{E}[d(v_i,t)]\\
							&=m\prod_{l=i}^{t-1}C(l,p,m)+\sum_{l=i}^{t-1}mp\, c_l\prod_{h=l+1}^{t-1}C(h,p,m) \\
							& \le O\left(t^{(1-p)/2}\right) +\left(\prod_{h=l_0+1}^{t-1}C(h,p,m)\right)\sum_{l=i}^{t-1}mp\, c_l \notag\\
							&= O\left(t^{(1-p)/2}\right) + O\left(t^{(1-p)/2}\right) O\left(\ln t^p\right)\\
							&= O\left(t^{(1-p)/2}\ln t^p\right). \notag
						\end{align*}
					\end{proof}
					
						\begin{proof}[\textbf{\textup{Proof of Theorem \ref{elvira}}}]
							Let us fix $s>1$. For $t\geq i+1$ define the stochastic process
							\begin{align}
								\label{stopro}
								S_{i,s}(t) &= \frac{d(v_i,t)}{\left(\sum_{l=i}^{t-1}\ln\left(1+c_l\right)\right)^s}\notag\\
								&+m\sum_{k=t}^{\infty}\frac{c_k}{\left(\sum_{l=i}^{k-1}\ln\left(1+c_l\right)\right)^s}.
							\end{align}		
							We claim $(S_{i,s}(t))_{t\geq i+1}$ is a non-negative supermartingale relative to $(\mathcal{F}_t)_{t\geq i+1}$.
							To see this, first observe that $c_k/\ln^sk=O\left(1/(k\ln^sk)\right)$
							and $\sum_{l=i}^{k-1}\ln\left(1+c_l\right)=O(\ln k)$.
							Now, since $s>1$,
							\begin{align*}
								\sum_{k=t}^{\infty}\frac{2^k}{2^k\left(\ln 2^k\right)^s}=\frac{1}{\ln^s 2}\sum_{k=t}^{\infty}\frac{1}{k^s}<\infty,
							\end{align*}
							and hence, according to the Cauchy condensation test we conclude that
							$\sum_{k=t}^{\infty}1/(k\ln^s k)<\infty$.
							Therefore, the series in \eqref{stopro} is convergent. Since $\mathbb{E}[d(v_i,t)]<\infty$ it follows that $\mathbb{E}[S_{i,s}(t)]<\infty$ for every $t\geq i+1$ as well. Clearly, $S_{i,s}(t)$ is $\mathcal{F}_t$-measurable for every $t\geq i+1$. Finally, recalling \eqref{expe}, $p=1$,
							\begin{align*}
								&\mathbb{E}[S_{i,s}(t+1)|\mathcal{F}_t]\\ &=\left(\sum_{l=i}^{t}\ln\left(1+c_l\right)\right)^{-s}\mathbb{E}[d(v_i,t+1)|\mathcal{F}_t]\\
								&+m\sum_{k=t+1}^{\infty}\frac{c_k}{\left(\sum_{l=i}^{k-1}\ln\left(1+c_l\right)\right)^s}\\
								& =\left(\sum_{l=i}^{t}\ln\left(1+c_l\right)\right)^{-s}\Bigl[d(v_i,t)\left(1-e_t\right)
								+m\,c_t\Bigr]\\
								&+m\sum_{k=t+1}^{\infty}\frac{c_k}{\left(\sum_{l=i}^{k-1}\ln\left(1+c_l\right)\right)^s} \notag\\
								& \leq \left(\sum_{l=i}^{t-1}\ln\left(1+c_l\right)\right)^{-s}\Bigl[d(v_i,t)
								+m\,c_t\Bigr]\\
								&+m\sum_{k=t+1}^{\infty}\frac{c_k}{\left(\sum_{l=i}^{k-1}\ln\left(1+c_l\right)\right)^s} \notag \\
								& =S_{i,s}(t), \notag
							\end{align*}
							as required. Therefore $(S_{i,s}(t))_{t\geq i+1}$ is a non-negative supermartingale. Consequently, there exist a non-negative random variable $X_{i,s}$ with $\mathbb{E}(X_{i,s})<\infty$ such that $S_{i,s}(t)\overset{a.s.}{\longrightarrow}X_{i,s}$ as $t\rightarrow \infty$.
							Plainly, the convergence of the series in \eqref{stopro} entails that for every $s>1$,
							\begin{align}
								\label{tazzina}
								\frac{d(v_i,t)}{\left(\sum_{l=i}^{t-1}\ln\left(1+c_l\right)\right)^{s}}\overset{a.s.}{\longrightarrow}X_{i,s} \qquad \text{as }t\rightarrow \infty.
							\end{align}
							Let now $s_0\in (1,s)$. Since
							\begin{align*}
								\left(\sum_{l=i}^{t-1}\ln\left(1+c_l\right)\right)^{-(s-s_0)}\overset{t\rightarrow \infty}{\longrightarrow}0,
							\end{align*}
							it follows that almost surely
							\begin{align*}
								&\frac{d(v_i,t)}{\left(\sum_{l=i}^{t-1}\ln\left(1+c_l\right)\right)^{s}}\\
								&=
								\frac{d(v_i,t)}{\left(\sum_{l=i}^{t-1}\ln\left(1+c_l\right)\right)^{s_0}}\frac{\left(\sum_{l=i}^{t-1}\ln\left(1+c_l\right)\right)^{s_0}}{\left(\sum_{l=i}^{t-1}\ln\left(1+c_l\right)\right)^{s}}\\
								&\longrightarrow X_{i,s_0}\cdot 0 \qquad \text{as }t\rightarrow \infty.
							\end{align*}
							This and \eqref{tazzina} immediately yield $X_{i,s}=0$ almost surely for every $s>1$, as claimed.
						\end{proof}
			
						\begin{proof}[\textbf{\textup{Proof of Theorem \ref{lntlnt}}}]
							First, for $t\geq i+1$ define $Z_i(t)=\exp(d(v_i,t))$. Clearly $\mathbb{E}[Z_i(t)]\leq \exp(2mt)$, for every $t\geq i+1$. Moreover,
							\begin{align}
								\label{rupu}
								&\mathbb{E}[Z_i(t+1)|\mathcal{F}_t]\notag\\
								& =\exp(d(v_i,t))\mathbb{E}[\exp(d(v_i,t+1)-d(v_i,t))|\mathcal{F}_t] \notag\\
								& = Z_i(t)\left(1-\frac{2mt+1-d(v_i,t)}{t[2mt+1-2m]}\right. \notag\\
								&+ \left.e\frac{2mt+1-d(v_i,t)}{t[2mt+1-2m]}\right)^m \notag\\
								& \leq Z_i(t)\left(1+(e -1)c_t\right)^m.
							\end{align}
							The second equality in \eqref{rupu} follows from the fact that
							\begin{align*}
								&d(v_i,t+1)-d(v_i,t)|\mathcal{F}_t\\
								&\sim \text{Bin}\left(m,\frac{2mt+1-d(v_i,t)}{t[2mt+1-2m]}\right).
							\end{align*}
							Then, for every $t\geq i+1$ define the positive supermartingale
							\begin{align*}
								W_i(t)=\frac{Z_i(t)}{\prod_{l=i}^{t-1}\left(1+(e-1)c_l\right)^m}
							\end{align*}
							and call $\eta_i$ the random variable with finite mean to which it converges almost surely. Recalling that $\ln x\leq x-1$ for every $x>0$ we obtain
							\begin{align*}
								\frac{d(v_i,t)}{\ln t}&=\frac{\ln(W_i(t))}{\ln t}
								+m\frac{\sum_{l=i}^{t-1}\ln\left(1+(e-1)c_l\right)}{\ln t} \notag \\
								&\leq \frac{W_i(t)-1}{\ln t}
								+m\frac{\sum_{l=i}^{t-1}\ln\left(1+(e-1)c_l\right)}{\ln t}, \notag
							\end{align*}
							for every $t\geq i+1$. Hence, recalling that $W_i(t)\overset{a.s.}{\longrightarrow}\eta_i\in [0,\infty)$,
							as $t\rightarrow \infty$,
							and using the fact that
							\begin{align*}
								&\sum_{l=i}^{t-1}\ln\left(1+(e-1)c_l\right)\leq (e-1)\sum_{l=i}^{t-1}c_l\\
								&=\sum_{l=i}^{t-1}\ln\left(1+(e-1)c_l\right)\leq (e-1)\sum_{l=i}^{t-1}O(1/l)\\
								&=O(\ln t),
							\end{align*}
							we finally obtain
							\begin{align*}
								\limsup_{t\rightarrow \infty}\frac{d(v_i,t)}{\ln t}<\infty \qquad \text{a.s.}
							\end{align*}
						\end{proof}
		
					\begin{proof}[\textbf{\textup{Proof of Theorem \ref{elvira2}}}]
						By Markov inequality, note first that,
						\begin{align*}
							&\mathbb{P}\left(\frac{d(v_i,t)}{\ln t} - \mathbb{E}\left[\frac{d(v_i,t)}{\ln t}\right]\ge \gamma(t)\right)\\
							&\le \exp \left[-\gamma(t) -
							\mathbb{E}\Bigl[\frac{d(v_i,t)}{\ln t}\Bigr]\right]	\mathbb{E}\left[\exp\left( \frac{d(v_i,t)}{\ln t}\right)\right].
						\end{align*}
						The last factor can be written as
						\begin{align*}
							&\mathbb{E}\left[\exp\left( \frac{d(v_i,t)}{\ln t}\right)\right]
							= \mathbb{E}\left[ \exp\left( \frac{d(v_i,t-1)}{\ln (t-1)}\right)\right.
							\\ &\times\left.\mathbb{E} \left(\left. \exp\left( \frac{d(v_i,t)}{\ln t} - \frac{d(v_i,t-1)}{\ln (t-1)}\right) \right| \mathcal{F}_{t-1} \right) \right] \\
							&\le \mathbb{E}\left[ \exp\left( \frac{d(v_i,t-1)}{\ln (t-1)}\right)\right.\\
							& \times \left. \mathbb{E} \left(\left. \exp\left(\frac{1}{\ln t} \bigl(d(v_i,t) - d(v_i,t-1)\bigr)\right) \right| \mathcal{F}_{t-1} \right) \right] \notag \\
							& = \mathbb{E}\left( \exp\left(\frac{d(v_i,t-1)}{\ln (t-1)}\right)\right.\\
							&\times\left.\left[ 1 + \frac{2m(t-1)+1-d(v_i,t-1)}{(t-1)[2m(t-1)+1-2m]}(e^{\frac{1}{\ln t}}-1) \right]^m \right) \notag \\
							& \le \exp\left[ m\,c_{t-1} (e^{\frac{1}{\ln t}}-1) \right] \mathbb{E}\left( \exp\left(\frac{d(v_i,t-1)}{\ln (t-1)}\right) \right), \notag
						\end{align*}
						which, by iteration gives
						\begin{align*}
							&\mathbb{E}\left[\exp\left( \frac{d(v_i,t)}{\ln t}\right)\right]\\
							&\le \exp \left[ \sum_{j=i+1}^{t-1}m \,c_j (e^{\frac{1}{\ln (j+1)}}-1) \right]
							\exp\left(\frac{m}{\ln i}\right) \\
							& = K \exp \left[ \sum_{j=i_0+1}^{t-1}m \,c_j (e^{\frac{1}{\ln (j+1)}}-1) \right], \notag
						\end{align*}
						where $i_0 = \min \{h \in \mathbb{N}^* \colon 1/\ln (h+1)< 1 \}$. Now, since $e^c-1 \le c+c^2$, $\forall \: c \in (0,1)$
						and therefore that
						\begin{align*}
							e^{\frac{1}{\ln (j+1)}} -1 \le \frac{1}{\ln (j+1)} + \frac{1}{\ln^2(j+1)} \le \frac{2}{\ln (j+1)},
						\end{align*}
						we readily obtain
						\begin{align*}
							&\mathbb{E}\left[\exp\left( \frac{d(v_i,t)}{\ln t}\right)\right]\\
							&\le K \cdot \exp \left[ \sum_{j=i_0+1}^{t-1} m \,c_j \frac{2}{\ln (j+1)} \right].
						\end{align*}
						Hence,
						\begin{align}
							\label{logi}
							&\mathbb{P}\left(\frac{d(v_i,t)}{\ln t} - \mathbb{E}\left[\frac{d(v_i,t)}{\ln t}\right]\ge \gamma(t)\right) \notag\\
							& \le K \exp \left[ - \gamma(t) -\mathbb{E} \left[ \frac{d(v_i,t)}{\ln t} \right] \right. \notag\\
							&+ \left. \sum_{j=i_0+1}^{t-1} m \,c_j \frac{2}{\ln (j+1)}  \right] \overset{t \to \infty}{\longrightarrow} 0.
						\end{align}
										
						Recalling Theorem \ref{sqrtno} and Remark \ref{loo}, for $t$ large enough, we write
						\begin{align*}
							\mathbb{E}\left[\frac{d(v_i,t)}{\ln t}\right]&< 2m \frac{\sum_{l=i}^{t-1}\ln\left(1+c_l\right)\prod_{h=l+1}^{t-1}\left(1-e_h\right)}{\ln t}\\
							&\le 2m \tilde{\kappa},
						\end{align*}
						where $\tilde{\kappa}$ is a suitable constant. Hence, for $\kappa>1$,
						\begin{align*}
							&\mathbb{P}(d(v_i,t) \ge \kappa \gamma(t) \ln t)\\
							& \le \mathbb{P}\big(d(v_i,t) \ge 2 m \tilde{\kappa} \ln t + \gamma(t) \ln t\big) \\
							& \le \mathbb{P} \left(d(v_i,t) \ge \Bigl (\mathbb{E}\Bigl[ \frac{d(v_i,t)}{\ln t} \Bigr] + \gamma(t) \Bigr) \ln t\right). \notag
						\end{align*}
						By using \eqref{logi} we conclude the proof.
					\end{proof}

	\subsubsection*{Acknowledgments} 
		
		Federico Polito and Laura Sacerdote have been partially supported by the project \emph{Memory in Evolving Graphs}
		(Compagnia di San Paolo/Universit\`a di Torino) and by INDAM--GNAMPA and INDAM--GNCS.
		Umberto De Ambroggio has been partially funded by the project \emph{Sviluppo e analisi di processi Markoviani e non Markoviani con applicazioni} (Università di Torino).

	\bibliographystyle{abbrvnat}
	\bibliography{Correct_VersionC}

\end{document}